\documentclass[11pt]{amsart}        
\usepackage{Mytheorems}
\usepackage{amssymb,amscd,amsthm,amsxtra}
\newtheorem{thm}{Theorem}[section]
\newtheorem{lem}[thm]{Lemma}
\newtheorem{rmk}[thm]{Remark}
\newtheorem{defi}[thm]{Definition}
\newtheorem{prop}[thm]{Proposition}
\usepackage{psfrag}
\usepackage{mathrsfs,amscd}
\newcommand{\parallelsum}{\mathbin{\!/\mkern-5mu/\!}}

\usepackage{graphicx}
\usepackage{epstopdf}

\usepackage{latexsym}
\input epsf

\usepackage{color}

\begin{document}
\title{ Regularity for diffuse reflection boundary problem to the stationary linearized Boltzmann equation in a convex domain  }
\author[IC]{I-Kun Chen}
\author{Chun-Hsiung Hsia}
\author{Daisuke Kawagoe}
\address{ }
\email{}
\date{\today}

\begin{abstract}
We investigate the regularity issue for the diffuse reflection boundary problem to the stationary linearized Boltzmann equation for hard sphere potential, cutoff hard potential, or cutoff Maxwellian molecular gases in a strictly convex bounded domain. We obtain  pointwise estimates for first derivatives of  the solution provided the boundary temperature is bounded differentiable and the solution is bounded. This result can be understood as a stationary version of the velocity averaging lemma and mixture lemma. 

\end{abstract}
\maketitle
\section{introduction}
In this article, we consider the stationary linearized  Boltzmann equation
\begin{equation}\label{SBE}
\zeta \cdot \nabla f(x,\zeta)=L(f),
\end{equation}
for $\zeta \in \mathbb{R}^3$ and $x\in\Omega$, where $\Omega \subset \mathbb{R}^3$ is a  $C^2$  bounded strictly convex domain such that $\partial \Omega$ is of positive Gaussian curvatures. Here, $L$  represents the linearization of the collision operator. The collision operator in Boltzmann equation reads:
\begin{align}\label{QFG} Q(F,G)&=\int_{\mathbb{R}^3} \int_0^{2\pi}\int_0^{\frac{\pi}2}(F(\zeta')G(\zeta_*')-F(\zeta)G(\zeta_*)B(|\zeta_*-\zeta|,\theta) d\theta d \epsilon d\xi_*,\end{align} where $\zeta$, $\zeta_*$ and $\zeta'$, $\zeta_*'$ are pairs of velocities before and after the impact, and $B$ is called the cross section, depending on interaction between particles.   $L$ is obtained by linearizing $Q$ around the standard Maxwellian
\begin{equation}M(\zeta)=\pi^{-\frac32}e^{-|\zeta|^2}\end{equation} in the fashion
\begin{equation}
F=M+M^{\frac12}f.
\end{equation}
$L$ reads
\begin{equation}
\label{Ldefine}
L(f)=M^{-\frac12}(\zeta)[Q(M^\frac12f,M)+Q(M,M^\frac12f)].
\end{equation}
The widely used angular cutoff potential is a mathematical model introduced by Grad \cite{Grad} by assuming
\begin{equation}
0\leq B \leq C|\zeta-\zeta_*|^\gamma\cos\theta\sin\theta.
\end{equation}
In this article,  we follow Grad's idea and  assume
\begin{equation}\label{Bsaperate}\begin{split}
&B=|\zeta-\zeta_*|^\gamma\beta(\theta),\\&
0\leq\beta(\theta)\leq C\cos\theta\sin\theta,\\&
0\leq\gamma\leq1.
\end{split}
\end{equation}
The range of $\gamma$ we consider corresponds  to the hard sphere model, cutoff hard potential, and cutoff Maxwellian molecular gases. We shall discuss the properties of $L$ under our assumption \eqref{Bsaperate} in detail in Section \ref{collisionoperator}.

The boundary condition under the consideration is
the diffuse reflection boundary condition :
\begin{enumerate}
\item[(1)] First,  there is no net flux on the boundary.
\item[(2)]Secondly, the velocity distribution function reflected from the boundary is in thermal equilibrium with the boundary temperature.
 \end{enumerate} In the context of the linearized Boltzmann equation, the mathematical formula of the aforementioned diffuse reflection boundary condition could be described as: for   $x\in\partial \Omega$ and $ \zeta\cdot n <0,$
\begin{align}
 \label{fbdry}
 f(x,\zeta)&=\psi(x)M^{\frac12}+T(x)(|\zeta|^2-2)M^{\frac12},\\
 \label{defpsi}
 \psi(x)&=2\sqrt{\pi}\int_{\zeta\cdot n > 0} f(x,\zeta)|\zeta\cdot n|M^\frac12d\zeta,
\end{align}
where $n=n(x)$ is the outward unit normal of $\partial \Omega$ at $x\in \partial \Omega$. Here, $T(x)$ is the temperature on the boundary.
 To state  our main goal of mathematical analysis, we define, for  given $x \in \bar{\Omega}$,  \begin{align}
\label{deftauminus}\tau_-(x,\zeta)&=\inf \Big\{t \, \Big|\, t> 0,\ x-t\zeta  \notin\Omega \Big\},  \text{ and }\\
\label{defpp}p(x,\zeta)&=x-\tau_-(x,\zeta)\zeta.
\end{align}
Under the assumption \eqref{Bsaperate}, $L$ can be decomposed into a multiplicative operator and an integral operator:
\begin{equation}
L(f)= - \nu(|\zeta|)f+K(f).\end{equation}
We take the integral operator $K$ as the source term and rewrite \eqref{SBE} as
\begin{equation}
\label{SBE2}
\zeta \cdot \nabla f(x,\zeta)+\nu(|\zeta|)f(x,\zeta)=K(f).
\end{equation}
The corresponding integral form of the solution to \eqref{SBE2} is  \begin{equation}\label{inteq}\begin{split}
f(x,\zeta)&=f(p(x,\zeta),\zeta)e^{-\nu(|\zeta|) \tau_-(x,\zeta)}+\int_0^{\tau_-(x,\zeta)}e^{-\nu(|\zeta|) s}K(f)(x-s\zeta, \zeta)ds.
\end{split}\end{equation}
We say $f$ is a solution to \eqref{SBE} if $f$ satisfies \eqref{inteq} almost everywhere. The existence of a solution to the presented  problem has been established by Guiraud in 1970 \cite{Guiraudlinear}. Recently, Esposito, Guo, Kim, and Marra extended the result to non-convex domains in  \cite{GuoKim}. In \cite{GuoKim}, under the assumption that $T(x)$ is bounded, the authors proved that the solutions to \eqref{SBE2} supplemented with the boundary conditions \eqref{fbdry}-\eqref{defpsi} is of function class $L^\infty_{x,\zeta}$. In the present article, we shall assume that $T(x)$ is  bounded differentiable, i.e., $T(x)$ is differentiable and its first derivatives are bounded, and we shall  aim at proving the interior differentiability of the solution to the problem \eqref{SBE2} supplemented with \eqref{fbdry}-\eqref{defpsi}, see Theorem~\ref{MainThm}.

It is worth mentioning that, in \cite{GuoKim}, they also proved that the solution is  continuous away from the grazing set.  For the higher regularity issue, by observing velocity averaging effect for the stationary linearized Boltzmann equation, the H\"older continuity up to $\frac12-$ away from the boundary was first established in \cite{RegularChen} for inflow boundary value problems.  In this article, we establish a pointwise estimate of the first derivatives of the solution. The main cruxes of this pointwise estimate are multifold. First, we need to overcome the difficulty brought by the diffuse reflection boundary condition. The diffuse reflection condition (\eqref{fbdry} and \eqref{defpsi}) involves the solution itself. Namely, inferring from \eqref{inteq} and \eqref{fbdry}-\eqref{defpsi}, since we only know that the solution $f \in L^\infty_{x,\zeta}$, we cannot even do the formal derivative to $f$ with respect to the space variable $x$. Secondly, we need to improve the regularity from H\"older continuity to differentiability.  We shall discuss these issues in depth after introducing the main theorem. Regarding regularity issues for the time  dependent Boltzmann equation, we refer the interested readers to \cite{GKTT, GKTT2, Kimdis}.

The main result of this article is as follows.
\begin{thm}\label{MainThm}
Let $\Omega \subset \mathbb{R}^3$ be a $C^2$ bounded strictly convex domain such that $\partial \Omega$ is of positive Gaussian curvatures. Under the assumption \eqref{Bsaperate}, suppose $f\in L^\infty_{x,\zeta}$ is a  solution to the stationary linearized Boltzmann equation \eqref{SBE} with the diffuse reflection boundary condition \eqref{fbdry}-\eqref{defpsi} such that
  the boundary temperature $T(x)$ $($in \eqref{defpsi}$)$ is bounded differentiable, i.e.,  $T(x)$ is differentiable and its first derivatives are bounded. Then, for $\epsilon>0$, we have
\begin{equation}
\sum_{i=1}^3 \left|\frac{\partial}{\partial x_i }f(x,\zeta) \right|+\sum_{i=1}^3 \left| \frac{\partial}{\partial \zeta_i} f(x,\zeta) \right|\leq C(1+d_x^{-1})^{\frac43+\epsilon}.\end{equation}
\end{thm}
Now we briefly recall some ideas about the velocity averaging effect for the stationary linearized Boltzmann equation on bounded convex domains introduced in \cite{RegularChen}. We iterate the integral equation \eqref{inteq} again and obtain \begin{equation}\begin{split}\label{def123}
&f(x,\zeta)=f(p(x,\zeta),\zeta)e^{-\nu(|\zeta|) \tau_-(x,\zeta)}\\
&+\int_0^{\tau_-(x,\zeta)}\int_{\mathbb{R}^3}e^{-\nu(|\zeta|) s}k(\zeta,\zeta')e^{-\nu(|\zeta'|) \tau_-(x-s\zeta ,\zeta')}f\big(p(x-s\zeta ,\zeta'),\zeta'\big)d\zeta'ds\\
&+\int_0^{\tau_-(x,\zeta)}\int_{\mathbb{R}^3}\int_0^{\tau_-(x-s\zeta ,\zeta')}e^{-\nu(|\zeta|) s}k(\zeta,\zeta')e^{-\nu(|\zeta'|) t}K(f)(x-s\zeta -t\zeta',\zeta')dtd\zeta'ds\\&=:I+II+III.
\end{split}\end{equation}
For $I$ and $II$, the regularity of the boundary is preserved by the transport. 
The velocity averaging effect would play an important role in the improvement of regularity for $III$. 
Notice that, thanks to nice property of the integral kernel, \eqref{gradianK}, in fact $K(f)$ is bounded differentiable in $\zeta$ provide $f$ is bounded, i.e.,
\begin{equation}
\left\Vert \frac{\partial}{\partial \zeta_i}K(f) \right\Vert_{L^\infty_\zeta}\leq C\Vert f \Vert_{L^\infty_\zeta}.\end{equation}
For time dependent kinetic equations defined on the whole space, it is well-known that velocity averaging effect combining with transport effect can transfer regularity from velocity variables to space variables, e.g., famous Velocity Averaging Lemma \cite{GolseAve} and Mixture Lemma \cite{LiuYu1}. As far as we know, there is no analogy result for stationary problems defined on bounded domains addressed elsewhere.
 To take care of the regularity of $III$,
we  change the variables $\zeta'$ to the spherical  coordinates so that
\begin{equation}
\zeta'=(\rho\cos\theta,\rho\sin \theta\cos\phi,\rho\sin \theta\sin \phi).
\end{equation} 
Also, we change the traveling time to the traveling distance:
\begin{equation}
r=\rho t.
\end{equation}
Let $\hat{\zeta'}=\frac{\zeta'}{|\zeta'|}.$
Then, we can rewrite $III$ as 
\begin{equation}
\begin{split}
III=&\int_0^{\tau_-(x,\zeta)}e^{-\nu(|\zeta|) s}\int_0^\infty\int_0^{\pi}\int_0^{2\pi}\int_0^{|x-s\zeta-p(x-s\zeta ,\zeta')|}\\
&\quad  k(\zeta,\zeta')e^{-\frac{\nu(\rho) }\rho r}K(f)(x-s\zeta -r \hat{\zeta'},\zeta')\rho\sin\theta drd\phi d\theta d\rho ds\\
=&: \int_0^{\tau_-(x,\zeta)}e^{-\nu(|\zeta|) s}G(x-s \zeta ,\zeta)ds.
\end{split}
\end{equation}
 Notice that  we can parametrize $\Omega$ by $\theta$, $\phi$, and $r$, thanks to the convexity of $\Omega$.  Therefore, by regrouping the integrals, we can change the formulation to contain an  integral over space:  Let $x_0=x-s\zeta $ and $y=x-s\zeta -r \hat{\zeta'}$. We have

\begin{equation}
\begin{split}
G(x_0,\zeta)=&\int_0^\infty\int_{\Omega} k \left(\zeta,\rho\frac{(x_0-y)}{|x_0-y|} \right) e^{-\nu(\rho)\frac{|x_0-y|}{\rho}}\\
& \qquad \qquad \times K(f) \left( y, \rho\frac{(x_0-y)}{|x_0-y|} \right)\frac{\rho}{|x_0-y|^2}dyd\rho.
\end{split}
\end{equation} 
Notice that in the above formula, the velocity variables $\zeta'$ are replaced by the space variables $x_0$ and $y$, and therefore the regularity in velocity variables can be transferred to space variables. However, the singularity in the above integral formula does not allow us to differentiate $G(x_0,\zeta)$ with respect to $x_0$ directly. This is the reason why the result obtained in  \cite{RegularChen} is limited to the H\"older type continuity.  
In the present work, we overcome this obstacle by bootstrapping the regularity from H\"{o}lder continuity to differentiability with the help of divergence theorem (see Section \ref{SDiffVDF} ). Notice that there is only a very narrow window that one can carry out  this strategy.  To this aim, we make  big efforts to significantly refine H\"{o}lder type  estimates in Sections \ref{HolderRV}, \ref{Mixtureenhanced}, and \ref{LocalHatBoundary} to make it.  In addition, we encounter not only the aforementioned cruxes but also the difficulties due to the diffuse reflection boundary condition.

Now, we shall give a brief account of the strategy that we employ to overcome all the aforementioned subtleties.
 First,  plugging   \eqref{fbdry} into \eqref{inteq}, we have
\begin{equation}
   \label{inteq1}
   \begin{split}
    f(x,\zeta)&=\big(\psi(p(x,\zeta))+ T(p(x,\zeta))(|\zeta|^2-2)\big)M^{\frac12} e^{-\nu(|\zeta|) \tau_-(x,\zeta)}\\
    &\,\,\,+\int_0^{\tau_-(x,\zeta)}e^{-\nu(|\zeta|) s}K(f)(x-s\zeta, \zeta)ds.
   \end{split}
\end{equation}
 Notice that, by \eqref{defpsi},  $\psi$ is only a bounded function provided $f$ is  bounded. Hence, we cannot even take a formal derivative on the first term of the right hand side of \eqref{inteq1}. To deal with the differentiability of the first term of the right hand side of \eqref{inteq1}, we only need to take care of the differentiability of $\psi(p(x, \zeta))$. 
Plugging \eqref{inteq1}  into \eqref{defpsi}, we obtain
\begin{equation}
\label{defphi3}
\begin{split}
\psi(x)=&2\sqrt{\pi}\int_{\zeta\cdot n(x) > 0}T(p(x,\zeta))(|\zeta|^2-2)M(\zeta)e^{-\nu(|\zeta|)\tau_-{(x,\zeta)}}|\zeta\cdot n(x)|d\zeta\\
&+2\sqrt{\pi}\int_{\zeta\cdot n(x) > 0}\psi(p(x,\zeta))M(\zeta)e^{-\nu(|\zeta|)\tau_-{(x,\zeta)}}|\zeta\cdot n(x)|d\zeta\\
&+2\sqrt{\pi} \int_{\zeta\cdot n(x) > 0}\int_0^{\tau_-(x,\zeta)}e^{-\nu(|\zeta|)s}K(f)(x-s\zeta, \zeta)M^{\frac12}(\zeta)|\zeta\cdot n(x)|dsd\zeta\\=:& B_T+B_\psi+D_f.
\end{split}\end{equation}
In the above expression, the domain of integration depends on the space variable $x$, which is not convenient for taking the formal derivatives. On the other hand, in the formula of $B_\psi$,  the integrand $\psi$ is a function of $p(x, \zeta)$. However, we only know that $\psi$ is a bounded function. 
As we shall go into the details in Section \ref{SecDiffBpsi}, by using a sophisticated change of variables, we can rewrite $B_{\psi}$ as follows
\begin{equation}
\label{bpsi02}
\begin{aligned}
B_\psi(x)=& \frac2\pi\int_0^\infty\int_{\partial\Omega}\psi(y)e^{-l^2|x-y|^2}e^{-\frac{\nu(l|x-y|)}{l}} \\
 & \qquad \times [(x-y)\cdot n(x)]|(x-y)\cdot n(y)|l^3dA(y)dl.
\end{aligned}
\end{equation}
We notice that with the above expression, the domain of integration is fixed and the formal derivative of  $B_\psi(x)$ with respect to $x$ variable does not involve $\psi$. This creates a startup for our regularity analysis. We use the same fashion to deal with the regularity of $B_T$.

 

For the term of $D_f$,  applying a  similar approach as in \cite{RegularChen}, we can convert $D_f$ into the following formula 
\begin{equation}
\begin{split}
D_f=2{\pi}^{-\frac14}\int_0^\infty\int_{\Omega}
& e^{-\frac{\nu(\rho)}{\rho}|x-y|}K(f) \left(y ,\rho\frac{(x-y)}{|x-y|} \right)\\
& \times  \frac{(x-y)\cdot n(x)}{|x-y|}e^{-\frac{\rho^2}2}\frac{\rho^2}{|x-y|^2}dyd\rho.
\end{split}
\end{equation}
Similar to the treatment of $III$, the advantage we gain from the above transformation is that the regularity in velocity variables can be transferred to space variables. However,  if we differentiate $D_{f}$ with respect to $x$ variables directly, the formula has a singularity which damages the integrability of the resulting formula. Therefore,  one can only claim the H\"older type  regularity by an argument similar to \cite{RegularChen}. Nevertheless, we can in fact further bootstrap the regularity to differentiability. For the details of the treatment, see Section~\ref{SDiffDf}.

The organization of the rest part of this article is as follows. We recapitulate the important properties of the linearized collision operator $L$ in Section~\ref{collisionoperator}. In Section~\ref{prelim}, we prepare several useful auxiliary lemmas and propositions associated to the geometry of $\Omega$ which play crucial roles in the integrability arguments in the estimates of Section~\ref{SecDiffBpsi} - Section~\ref{SDiffVDF}. In Section~\ref{SDiffVDF}, we sum up the estimates from Section~\ref{SecDiffBpsi} to Section~\ref{SDiffDf} and conclude the differentiability in $x$ variables. Section~\ref{zetaDiff} is devoted to the differentiability in $\zeta$ variables.


\section{Properties of linearized collision operator \label{collisionoperator}}

In this section, we  summarize some known properties of the linearized collision operator $L$ defined in \eqref{Ldefine} with a
cross section satisfying our assumption \eqref{Bsaperate}
 (See \cite{Caflisch, ChenHsia, Grad}). $L$ can be decomposed into a multiplicative operator and an integral operator: \begin{equation}
L(f)=-\nu(|\zeta|)f+K(f),
\end{equation}
where 
\begin{equation}
\label{defK}
K(f)(x,\zeta)=\int_{\mathbb{R}^3}k(\zeta,\zeta_*)f(x,\zeta_*)d\zeta_*\end{equation} is symmetric, i.e.,
$$
k(\zeta,\zeta_*) = k(\zeta_*, \zeta).
$$

  The explicit expression of $\nu$ is
\begin{equation}
\nu(|\zeta|)=\beta_0\int_{\mathbb{R}^3}e^{-|\eta|^2}|\eta-\zeta|^\gamma d\eta,
\end{equation}
where $\beta_0=\int_0^{\frac{\pi}2}\beta(\theta)d\theta$. Let $0<\delta<1$. The collision frequency $\nu(|\zeta|)$ and the collision kernel $k(\zeta,\zeta_*)$ satisfy
\begin{align}
\label{estimatenu}
&\nu_0(1+|\zeta|)^\gamma\leq \nu(|\zeta|)\leq \nu_1(1+|\zeta|)^\gamma,\\
 \label{estimateK}
&|k(\zeta,\zeta_*)|\leq C_1|\zeta-\zeta_*|^{-1}(1+|\zeta|+|\zeta_*|)^{-(1-\gamma)}{e^{-{\frac{1-\delta}4}\left(|\zeta-\zeta_*|^2+(\frac{|\zeta|^2-|\zeta_*|^2}{|\zeta-\zeta_*|})^2\right)}},\\
 \label{gradianK}&\left| \frac{\partial}{\partial \zeta_i} k(\zeta,\zeta_*) \right|\leq C_2\frac{1+|\zeta|}{|\zeta-\zeta_*|^{2}}(1+|\zeta|+|\zeta_*|)^{-(1-\gamma)}{e^{-{\frac{1-\delta}4}\left(|\zeta-\zeta_*|^2+(\frac{|\zeta|^2-|\zeta_*|^2}{|\zeta-\zeta_*|})^2\right)}}.
\end{align}
Here,  the constants $0<\nu_0<\nu_1$ may depend on the potential and  $C_1$ and $C_2$  may depend on $\delta$ and the potential. 
Notice that \eqref{estimateK} was established in \cite{Caflisch} and \eqref{gradianK} can be concluded by the observation in \cite{ChenHsia} in case the cross section  satisfies \eqref{Bsaperate}.

Related to the above estimates,  the following proposition from \cite{Caflisch} is crucial in our study.
\begin{prop} \label{cafdecay} For any $\epsilon, a_1, a_2>0$,
\begin{equation}
\Big|\int_{\mathbb{R}^3} \frac1{|\eta-\zeta_*|^{3-\epsilon}}e^{-a_1|\eta-\zeta_*|^2-a_2\frac{(|\eta|^2-|\zeta_*|^2)^2}{|\eta-\zeta_*|^2}}d\zeta_*\Big|\leq C_4 (1+|\eta|)^{-1}, \end{equation}
\noindent where $C_4$ may depend on $\epsilon, a_1,$ and $a_2$.
\end{prop}

Using the \eqref{gradianK} and Proposition \ref{cafdecay}, we can conclude 
\begin{equation}
\label{CH}
\left\Vert \frac{\partial}{\partial \zeta_i} k(\zeta,\zeta_*) \right\Vert_{L^\infty_\zeta L^1_{\zeta_*}}<\infty,\ \left\Vert \frac{\partial}{\partial\zeta_i} k(\zeta,\zeta_*) \right\Vert_{L^\infty_{\zeta_*} L^1_{\zeta}}<\infty. \end{equation} Then, by Schur's test, we can conclude the following smooth effect of $K$ in velocity variable mentioned in \cite{ChenHsia}.
\begin{prop}
For $1\leq p\leq \infty$,\begin{equation}
\left\Vert \frac{\partial}{\partial \zeta_i} K(f) \right\Vert_{L^p_\zeta}\leq C\Vert f \Vert_{L^p_\zeta}.\end{equation}

\end{prop}

\section{Geometric Properties \label{prelim}}
In this section, we shall prove several important auxiliary lemmas and propositions  which play  important roles in our regularity theory.
\begin{lem}
\label{lemmaA}
Suppose $\Omega$ is a $C^2$ bounded strictly convex domain, whose boundary is  of positive Gaussian curvature. 
Then, there exists a constant  $C$  depending only on the geometry of domain $\Omega$ such that for any interior point $x \in \Omega$, we have
\begin{equation}\label{inversesquareint}
\int_{\partial\Omega}\frac1{|x-y|^2}dA(y)\leq C(\big|\ln d_x \big|+1),
\end{equation}
where  $d_x = d(x, \partial \Omega)$ and  $A(y)$ is the surface element of $\partial \Omega$ at point $y \in \partial \Omega$.
\end{lem}
 Lemma~\ref{lemmaA} is crucial in the proof of the refined H\"{o}lder type estimate \eqref{kdIest} in Lemma~\ref{HolderIandII} as well as Proposition~\ref{Hprop}.
A special regular case of Lemma~\ref{lemmaA} is the case where $\partial \Omega$ is a sphere. It is an extreme case that $\partial \Omega $ is a bounded subset of $ \mathbb{R}^2$.  One may prove Lemma~\ref{lemmaA} for the above two cases by direct calculation. To deal with the general case, we need the following proposition.
\begin{prop}
\label{propositionA}
Suppose $\Omega$ is a $C^2$ bounded strictly convex domain, whose boundary is  of positive Gaussian curvature. 
  Then, there exists a constant $r_1$ $($ see \eqref{defr} $)$ depending only on $\Omega$ such that for any $x \in \Omega$ and $p_0 \in \partial \Omega$ satisfying that $(p_0-x)$ is parallel to $n(p_0)$,where $n(p_0)$ is the unit outward normal of $\partial \Omega$ at $p_0$, there holds the following inequality
 \begin{align}\label{expmapineq}
 & |x-p_0|^2+\frac12|v|^2\leq |Exp_{p_0}(v)-x|^2,\\
\nonumber  &\mbox{ for }0 \leq |x-p_0| \leq r_1 \text{ and }  v \in T_{p_0}(\partial \Omega) \text{ with } \ 0\leq |v| \leq r_1.
  \end{align}
Here, $Exp_{p_0}$ is the exponential map from the tangent space $T_{p_0}(\partial \Omega)$ to $\partial \Omega$.

 \end{prop}

 \begin{proof}
 Our assumptions on $\Omega$ imply there are uniform positive upper and lower bounds for normal curvature and Gaussian curvature of $\partial \Omega$.  By \cite{cheeger},
there is a uniform radius, $r_0$, and a positive constant $a_0 < 1$
 such that for every point $p \in \partial \Omega$  the exponential map $Exp_{p}: T_p(\partial \Omega) \to \partial \Omega$ is one-to-one and the Jocobian satisfies
 \begin{equation} 
 \label{jacobian}
 a_0 \leq \Big|\det{ \left(\frac{\partial Exp_p}{ \partial X} \right)}\Big| \leq 1
\end{equation} 
  within the $r_0-$neighborhood of $T_p(\partial \Omega)$. We are going to estimate the distance between $x \in \Omega$ and a point in the geodesic disc centered at $p_0$.   We choose the coordinate
such that $p_0=(0,0,0)$, $x=(0,0,-d)$. Without loss of generality, we  only need to consider the points on the normal geodesic $\phi(s)=(\phi_1(s),\phi_2(s),\phi_3(s))$ with $\phi(0)=(0,0,0)$ and $\phi'(0)=(1, 0, 0)$.
 Since normal curvature is bounded, there exist constants $0< a < b$ independent of $p_0$ and $\phi$ such that
 \begin{equation}
 \label{normalcurvature}
 0<a\leq \left| \frac{d^2}{ds^2}\phi(s) \right|\leq b,
 \end{equation}
 for all $s \in (-\delta, \delta)$.
By $\phi'(0)=(1, 0, 0)$,  we derive from \eqref{normalcurvature}  that
 \begin{align}
\label{phiprime1}1-bs &\leq \phi_1'(s)\leq 1+bs,\\ 
\label{phiprime2}-bs&\leq  \phi_2'(s)\leq bs,\\
\label{phiprime3} -bs&\leq  \phi_3'(s)\leq bs.
 \end{align}
 Therefore,
  \begin{align}
\label{phi1}s-\frac12bs^2 &\leq \phi_1(s)\leq s+\frac12bs^2,\\
\label{phi2} -\frac12bs^2&\leq  \phi_2(s)\leq \frac12bs^2,\\ 
\label{phi3}-\frac12bs^2&\leq  \phi_3(s)\leq \frac12bs^2.
 \end{align}
 For further discussion, we define 
 \begin{equation}
 \label{defr}
 r_1 := \min \Big\{ r_0, \frac{1}{4b} \Big\}. 
 \end{equation}
 In the following analysis, we assume that $$ 0 \le s \le r_1,  \text{ and }  0< d \le r_1. $$
 
  \textit{Case 1: $d > \frac{1}{2}bs^2$.}
  In this case, we have  
  \begin{equation*}
  \phi_3(s)+d  >  d -\frac{1}{2} b s^2,
\end{equation*}
and  
\begin{align*}
|\phi(s)- x|^2 & \ge \phi_1^2(s) + \phi^2_2(s) + \left( d - \frac{1}{2}bs^2 \right)^2\\
&\ge \left(s - \frac{1}{2}bs^2 \right)^2 + 0^2 + \left( d -\frac{1}{2}bs^2 \right)^2 \\
&= d^2 + (1-bd -bs)s^2 + \frac{1}{2}b^2 s^4 \\
& \ge d^2 + \frac{1}{2} s^2. 
\end{align*}
 \textit{Case 2: $d \le  \frac{1}{2}bs^2$.} In this case, we see that
\begin{align*}
|\phi(s)- x|^2 & \ge \phi_1^2(s) \\
& \ge \left( s- \frac{1}{2}bs^2 \right)^2\\
&= \left( \frac{1}{2}bs^2 \right)^2 + (1-bs)s^2\\
& \ge d^2 +\frac{1}{2}s^2.
\end{align*} 
Now, for   $ v \in T_{p_0}(\partial \Omega) \text{ with } \ 0\leq |v| \leq r_1,$ by choosing the coordinate properly, we have $\phi(s)=Exp_{p_0}(v)$ with $s = |v|.$ Summing up  \textit{Case 1} and \textit{Case 2}, we conclude Proposition~\ref{propositionA}.
\end{proof}
\begin{rmk}
\label{rem32}
Taking \eqref{defr} into account, by \eqref{phi1}, we see that 
\begin{equation}
s \le \frac{8}{7} |\phi(s) - \phi(0)|.
\end{equation}
 
\end{rmk}

\textit{Proof of Lemma~\ref{lemmaA}.}
  We use the notation
\begin{equation}
GB(p,r):=\big\{ Exp_{p}(v)\big| |v|<r \big\}
\end{equation}
to denote the geodesic disc on $\partial \Omega$ centered at $p$ with geodesic radius $r.$ 
We first take care of the case where $d_x\leq {r_1}$.
We define
 \begin{align}&D_0:=\big\{p\in \partial\Omega\big|(p-x)\parallelsum n(p),\ d_x\leq|p-x|\leq r_1 \big\},\\
 &D_1:=\bigcup_{p \in D_0}GB \left( p, \frac{1}{10}r_1\right).\end{align}
By the Vitali's covering lemma, we see that there exists a countable subcollection  $\tilde{D}_0$  of $D_0$ such that $ \bigcup_{p \in \tilde{D}_0 }GB(p, \frac{1}{10}r_1) $ is a disjoint union of  geodesic discs $GB(p, \frac{1}{10}r_1)$  and  
\begin{equation}
    D_1      \subset \bigcup_{p \in \tilde{D}_0}GB \left(p, \frac{5}{10}r_1 \right). 
\end{equation}
 On the other hand, due to \eqref{jacobian}, there is a uniform lower bound $A_1$  of the area of $GB(p, \frac{1}{10}r_1)$ for any $p \in \partial \Omega$. Since $\bigcup_{p \in \tilde{D}_0 }GB(p, \frac{1}{10}r_1)$ is a disjoint union of $GB(p, \frac{1}{10}r_1)$ and the area of $\partial \Omega$ (denoted by $A_2$) is finite, the cardinality  of $\tilde{D}_0$ satisfies
\begin{equation}
\#(\tilde{D}_0) \le \frac{A_2}{A_1} < \infty. 
\end{equation} 
 We remark that, by the above argument, the upper bound of the cardinality $\#(\tilde{D}_0)$ is independent of  the position of $x.$ For the sake of convenience, we list all the elements of  $\tilde{D}_0$ as follows
 \begin{equation}
 \tilde{D}_0 = \Big\{p_1, p_2, \cdots, p_m \Big\},
 \end{equation}
 where $m = \#(\tilde{D}_0)$, and define the sets
\begin{align*}
D_2 & := \bigcup_{i=1}^{m}GB(p_i,r_1)   \text{ and } \\
D_3 & := \partial \Omega\setminus D_2.
\end{align*}
Note that since $D_2$ is an open subset of $\partial \Omega$, $D_3$ is a compact set. Hence, there is a point $p'\in D_3$ that realizes the distance of $x$ and $\partial \Omega.$ 
 
 We claim that \begin{equation}
 |p'-x|= d(x, D_3) > \frac{r_1}{\sqrt{2}}.\end{equation}
 Suppose, on the contrary,
 $|p'-x|=d(x, D_3) \le \frac{r_1}{\sqrt{2}}$.
We first observe that if $p'$ is an interior point of $D_3$, due to  $|p'-x|=d(x, D_3)$, we see that $(p'-x)\parallelsum n(p')$ and hence $p'\in D_0$. This violates to $p'\in D_3$. On the other hand, if $p' \in \partial D_3$, then there exists $p_j \in \tilde{D}_0$ such that $p' \in \partial GB(p_j, r_1).$ Therefore, there exists $v \in T_{p_j}(\partial \Omega)$ with $|v|=r_1$ such that $Exp_{p_j}(v)=p'.$ Applying Proposition~\ref{propositionA}, we derive that
\begin{equation}
|x-p_j|^2 + \frac{1}{2}r_1^2 \le |p'-x|^2 \le \frac{1}{2}r_1^2. 
\end{equation}
This implies $x = p_j$ which is a contradiction. We then conclude the claim. 

It is easy to see that
\begin{equation}
\label{bigthanr}
\int_{D_3}\frac{1}{|x-y|^2}dA(y)\leq \frac2{r_1^2}\int_{\partial \Omega} dA(y)\leq \frac{2A_2}{r_1^2}.\end{equation}
On the other hand,
\begin{equation}
\label{smallthanr}
\begin{split}
&\int_{GB(p_i, r_1)}\frac{1}{|x-y|^2}dA(y)\leq \int_0^{r_1}\int_0^{2\pi}\frac1{|x-p_i|^2+\frac12s^2}  sdsd\theta\\
&\le 2\pi \int_0^{\frac12 r_1^2} \frac1{d_x^2+u}du \\ 
& \leq 4\pi\big|\ln d_x \big|+2\pi \left|\ln \left( d_x^2+\frac{1}{2}r_1^2 \right) \right|\\&\leq C \left(1+\big|\ln d_x\big| \right).\end{split}\end{equation}
Taking \eqref{bigthanr} and \eqref{smallthanr} into account, we prove Lemma~\ref{lemmaA} for the case where $d_x \le r_1.$
For the case where $d_x > r_1$, we may bound the left hand side of \eqref{inversesquareint} by
the right hand side of \eqref{bigthanr}. This completes the proof of  Lemma~\ref{lemmaA}. 
\qed

Next, we investigate the estimates given in the following lemma.

\begin{lem}
\label{lemmaB}
Adopting the same geometric assumption on $\Omega$ as stated in Lemma~\ref{lemmaA}, let $r_1$ be as defined by \eqref{defr} in Lemma~\ref{lemmaA}. Then, there exists a constant $C$ such that for any  $x, y \in \partial \Omega$ with $y \in GB(x,r_1)$, we have
\begin{align}
\label{lemmaB1}&|n(x)\cdot(x-y)|\leq C|x-y|^2,\\
\label{lemmaB2}&|n(y)\cdot(x-y)|\leq C|x-y|^2,\\
\label{lemmaB3}&|n(y)\cdot v|=|n(y)\cdot (v-v')|\leq C|x-y|,
\end{align}
where $v \in T_x(\partial \Omega)$ is a unit vector and  $v' \in T_y(\partial \Omega)$ is the parallel transport of $v$ from $T_x(\partial \Omega)$ to $T_y(\partial \Omega)$.
\end{lem}  

The above lemma is crucial in the proofs of the Lemma~\ref{difBpsi} and Lemma~\ref{Dfdiff}.   These geometric observations can resolve the  difficulty from seemingly critical singularity (barely non-integrable) on a surface encountered in the proofs.   
\begin{proof}
By choosing an appropriate coordinate system, we may assume $x=(0,0,0)$, $n(x)=(0,0,1)$ and $\phi(s)$ is the normal geodesic on $\partial \Omega$ connecting $x$ and $y$ within the geodesic disc $GB(x, r_1)$ such that 
\begin{equation*}
\left\{
\begin{aligned}
& \phi(0) = x,\\
& \phi'(0)=(1,0,0),\\
& Exp_{x}((\tau, 0,0))=\phi(\tau)=y.
\end{aligned}
\right.
\end{equation*} 
Replacing $s$ by $\tau$ in the estimates \eqref{phi1}-\eqref{phi3}, we obtain 
\begin{align*}
& |x-y|\ge |\phi_1(\tau)| \ge \tau - \frac{1}{2}b\tau^2 \ge (1 - \frac{1}{2}b r_1)\tau \ge \frac{1}{2} \tau,   \,\,  \text{ and } \\
&|n(x)\cdot(x-y)|= |\phi_3(\tau)|\le \frac{1}{2}b \tau^2 \le 2b |x-y|^2.
\end{align*}
 This proves \eqref{lemmaB1}. By symmetry, \eqref{lemmaB2} is derived from \eqref{lemmaB1}.
To see \eqref{lemmaB3}, by replacing $s$ by $\tau$ in \eqref{phiprime1}-\eqref{phiprime3}, we obtain 
\begin{align*}
|n(y)\cdot v| &=|n(y)\cdot (v-v')| \le |v-v'|\\
              &=|\phi'(0)-\phi'(\tau)| \le \sqrt{3}b\tau \\
              &\le 2 \sqrt{3}b |x-y|.
\end{align*}
Finally, we may choose $C=4b$ so that \eqref{lemmaB1}-\eqref{lemmaB3} hold true.
\end{proof}

The next lemma is an important ingredient of the proof  of the  H\"older type estimate up to the boundary, Lemma~\ref{HolderBoundary}. 

\begin{lem}\label{Atboundarysqrt}Let $\Omega$ be a $C^2$  bounded strictly convex domain in $\mathbb{R}^3$ such that $\partial\Omega$ is of positive Gaussian curvature. Then, there exists $R_0 > 0$ depending only on $\Omega$ such that  if $x\in\partial\Omega$, $y \in \Omega$, and
   \begin{align} \label{Nearboundary}
 d_y \leq R_0,
\end{align} 
then a point $Y\in\partial \Omega$ such that $d(Y,y)=d_y$ is unique. Furthermore, there exist $C_1'$, $C_2'>0$ such that, for $y \in \Omega$ satisfying $(\ref{Nearboundary})$ and $x \in \partial \Omega$, if 
\begin{equation}
n(Y)\cdot(x-y)\geq 0,
\end{equation}
then \begin{equation}\label{sqrtgeq}
|x-y|\leq C_1' d_y^{\frac12}, \end{equation}
or if \begin{equation}
n(Y)\cdot(x-y)\leq 0,\end{equation}
then\begin{equation}
|x-y|\geq C_2'd_y^{\frac12}.\end{equation}

\end{lem}

\begin{proof} It is well known that there exists $R_1>0$ such that $d_y\leq R_1$ implies the existence of unique projection $Y$ on $\partial \Omega$.  The important task is to prove the second part of the lemma. Because of the assumption on $\Omega$, there exist $R_3>R_2>0$ such that for every point $p$ on $\partial \Omega$ there exist a sphere $ S_o(p)$ with radius $R_3$  and a sphere $S_i(p)$ with radius $R_2$ both tangent to $\partial \Omega$ at $p$ and $S_o(p)$ contains the whole $\Omega$ and  while $S_i(p)$ is contained completely within $\Omega$.  We let $R_0=\min(R_1,R_2)$ and consider $y$ with $d_y \leq R_0$. We name the centers  of $S_o(Y)$ and $S_i(Y)$ as $Y_o$ and $Y_i$ respectively.  Also, we name the plane perpendicular to $n(Y)$ passing $y$ as $L$.  $L$ divides $\partial\Omega$ into two components.  If $n(Y)\cdot(x-y)\geq 0$ then $x$ and $Y$ fall in the same component. $\overrightarrow{yx}$ intersects $S_o(Y)$ at one point $X'$. Let $L'$ be the plane passing $x$, $y$, and $Y_o$. There are two intersection points among $L$, $L'$, and $S_o(Y)$. We name the one closer to $X'$ as $A$. We can observe\begin{equation} |x-y|\leq |X'-y|\leq |A-y|.\end{equation}

Let $\theta=\angle {AY_oy}$. Then, \begin{equation}
\cos\theta=\frac{|y-Y_o|}{|A-Y_o|}=\frac{R_3-d_y}{R_3}=1-\frac{d_y}{R_3}.\end{equation}
Therefore,
\begin{equation}
\frac{d_y}{R_3}=1-\cos\theta=2\sin^2{\frac{\theta}2}.
\end{equation} 
We obtain\begin{equation}
\sin\frac{\theta}2=\sqrt{\frac{d_y}{2R_3}}. \end{equation}On the other hand, we have
\begin{equation} |A-y|=R_3\sin\theta=2R_3\sin\frac{\theta}2\cos\frac{\theta}2\leq \sqrt{2R_3}\sqrt{d_y}.
\end{equation}
We finished the proof of \eqref{sqrtgeq}. If $n(Y)\cdot (x-y)\leq0$, then $x$ and $Y$ are on different components. Let $B$ be the intersection point between $\overline{xy} $ and $S_i(Y)$. 
We have\begin{equation}|x-y|>|y-B|.\end{equation} We name the plane passing $y$, $Y_i$, and $x$ as $L_1$ and the plane perpendicular to $n(Y)$ passing $Y_i$ as $L_2$. 
There are two intersection points among  $L$, $L_1$, and $S_i(Y)$.
We denote the intersection point which is closer to $B$  as $U$. If $B$ lays between $L$ and $L_2$, then \begin{equation}\angle{ BUy}\geq \frac{\pi}2.\end{equation} Therefore,\begin{equation}|B-y|\geq |U-y|.\end{equation}  Let $\theta'=\angle UY_iy$. Similarly, we have\begin{equation}
\sin \frac{\theta'}2=\sqrt{\frac{d_y}{2R_2}}\end{equation}
Therefore,\begin{equation}\begin{split}|y-U|&=R_2\sin\theta'=2R_2\sin\frac{\theta'}2\cos\frac{\theta'}2\\&\geq\sqrt{2}R_2\sin\frac{\theta'}2\geq\sqrt{R_2d_y}.\end{split}\end{equation}
For the case $Y$ and $B$ are on different side of $L_2$, \begin{equation}|y-B|\geq R_2\geq\sqrt{R_2d_y}.\end{equation}  Hence, we finish the proof. \end{proof}

\section{Differentiability of $B_\psi$ and $B_T$ \label{SecDiffBpsi}}
By \eqref{defpsi}, the definition of $\psi$, we  see that $\psi$ is bounded whenever $f\in L^\infty_{x,\zeta}$.  In this section, we shall further prove that the first derivatives of $B_T$ and $B_\psi$  are bounded provided $T$ and $\psi$ are bounded. By the differentiability on the boundary of $\Omega$, we refer to the directional derivatives:
\begin{defi}
Let  $ x, \eta\in \mathbb{R}^3 $ and $D$ be a $C^1$ surface in $\mathbb{R}^3$ and $ f:D\subset\mathbb{R}^3\to \mathbb{R}$. Suppose $\phi:(-\epsilon,\epsilon)\to D$ is a smooth space curve such that
\begin{equation}
\phi(0)=x,\ \ \left.\frac{d}{d t} \phi (t) \right|_{t=0}=\eta.
\end{equation}We define
 \begin{equation}
\nabla^x_{\eta}f(x):=\left.\frac{d}{dt} f( \phi(t))\right|_{t=0}\end{equation}
when the limit at right-hand-side exists.
  \end{defi}
  
Our first result in this section is the following lemma.
\begin{lem} 
Suppose $\Omega$ satisfies the same assumption in the Main Theorem. Suppose $T(x)$ and $\psi(x)$ are bounded. Then, the first derivatives of $B_T(x)$ and $B_\psi(x)$ are bounded.\label{difBpsi}
\end{lem}

Recall that
\begin{align}
\label{bt} B_T(x)&:=2\sqrt{\pi}\int_{\zeta\cdot n > 0}T(p(x,\zeta))(|\zeta|^2-2)M(\zeta)e^{-\nu(|\zeta|)\tau_-{(x,\zeta)}}|\zeta\cdot n|d\zeta, \\
\label{bpsi1}B_\psi(x)&:= 2\sqrt{\pi}\int_{\zeta\cdot n > 0}\psi(p(x,\zeta))M(\zeta)e^{-\nu(|\zeta|)\tau_-{(x,\zeta)}}|\zeta\cdot n|d\zeta.
\end{align}

We shall only present the proof for $B_\psi$ because the proof for $B_T$ is similar.
The following proposition gives a useful alternative formulation of  $B_\psi$.
\begin{prop}   
\begin{equation}
\label{bpsi2}
B_\psi(x)=\frac2\pi\int_0^\infty\int_{\partial\Omega}\psi(y)e^{-l^2|x-y|^2}e^{-\frac{\nu(l|x-y|)}{l}}[(x-y)\cdot n(x)]|(x-y)\cdot n(y)|l^3dA(y)dl.
\end{equation}
\end{prop}
\begin{proof}
The idea of showing the equivalence between \eqref{bpsi1} and \eqref{bpsi2} is to do a change of coordinates.   We first observe that, by the strictly convexity of $\Omega$, for each $\zeta$ in the half space    
$$H= \Big\{ \zeta \in \mathbb R^3 \Big| \zeta \cdot n(x)>0 \Big\},$$
there exists exactly a unique pair $(y, l) \in \partial \Omega \times \mathbb{R}_+$ such that
\begin{equation}
\label{cone1}
\zeta = l(x - y).
\end{equation}
Secondly, since the bounded set $\Omega$ is $C^2$ strictly convex, we can cover $\partial \Omega$ by finitely many local charts, i.e., for $1\leq  i\leq k$, there are 
\begin{equation*}
\left\{
\begin{aligned}
& \text{ simply-connected open set }   D_i\subset \mathbb{R}^2,  \,\, \text{ and }  \\
& C^2  \text{ local diffeomorphism }  \phi_i:D_i\to \partial\Omega
\end{aligned}
\right.
\end{equation*}
such that 
\begin{equation*}
\bigcup_{1\leq i\leq k} \phi_i(D_i)=\partial\Omega.
\end{equation*}
Summing up from the above observations, we can parametrize the half space $H$ by a union of a finite number of cone domains. That is to plug $y = \phi_i(\alpha, \beta)$ into \eqref{cone1}. This gives a coordinate change
$\zeta_i:D_i\times(0,\infty)\to H$ :
\begin{equation}
\zeta_i(\alpha,\beta,l)=l\big(x-\phi_i(\alpha,\beta)\big).
\end{equation}
Denote $H_i:=\zeta_i(D_i\times (0,\infty))$. We see that $H = \bigcup_{i=1}^{k} H_i.$
On the other hand, direct calculation shows that the Jocobian of this coordinate change is given by
\begin{equation}
\Big|\frac{\partial \zeta_i (\alpha,\beta, l)}{\partial (\alpha,\beta, l)}\Big|=l^2\Big|\big(x-\phi_i(\alpha,\beta)\big)\cdot \big(\partial_\alpha\phi_i\times\partial_\beta\phi_i\big)\Big|.
\end{equation}
We readily see that
\begin{equation}
\begin{split}
&2\sqrt{\pi}\int_{H_i} \psi(p(x,\zeta)) M(\zeta)e^{-\nu(|\zeta|)\tau_-{(x,\zeta)}}|\zeta\cdot n|d\zeta\\
=&\frac2\pi\int_0^\infty\int_{D_i}\psi(\phi_i(\alpha,\beta))e^{-l^2|x-\phi_i(\alpha,\beta)|^2}e^{-\frac{1}l\nu(l|x-\phi_i(\alpha,\beta)|)}\\
&\quad\times[(x-\phi_i(\alpha,\beta))\cdot n(x) ]|(x-\phi_i(\alpha,\beta))\cdot[\partial_\alpha\phi_i\times\partial_\beta\phi_i]|l^3d\alpha d\beta\\
=&\frac2\pi\int_0^\infty\int_{\phi_i(D_i)}\psi(y)e^{-l^2|x-y|^2}e^{-\frac{\nu(l|x-y|)}l}[(x-y)\cdot n(x) ]|(x-y)\cdot n(y)|l^3dA(y)dl,
\end{split}
\end{equation}
where $y= \phi(\alpha, \beta)$, $n(y)$ is the outward unit normal of $\partial \Omega$ at $y$ and $A(y)$ is the surface element of $\partial \Omega$ at $y$.

Combining all the pieces and excluding the repetitions, we obtain the desired formula.
\end{proof}

  \begin{defi} For $x, y\in \Omega$ and $\zeta\in \mathbb{R}^3,$ we define  
  
   \begin{align}
   \tau_{-}(x,\zeta)&:=\inf\{t>0|x-t\zeta\notin \Omega \},\\
   p(x,\zeta)&:=x-\tau_-(x,\zeta)\zeta,\\
   d_x&:=\inf \big\{|x-y| \big|y\in\partial \Omega\big\},\\
   d_{x,y}&:=\min \{ d_x, d_y  \},\\
   N(x,\zeta)&:=\frac{|n(p(x,\zeta))\cdot\zeta|}{|\zeta|}.   \end{align}
   \end{defi}

Now, we are ready to prove the Lemma \ref{difBpsi}
\begin{proof}[Proof of Lemma \ref{difBpsi}]
Let  $\alpha(t)$ be a normal geodesic and  $v\in T_x (\partial \Omega)$, $|v|=1$   such that 
\begin{equation}
\begin{split}
&\alpha(0)=x,\\
&\left.\frac{d}{dt}\alpha(t)\right|_{t=0}=v.
\end{split}
\end{equation}
Then
\begin{equation}
\begin{split}
&\nabla^x_{v}B_\psi(x) :=\left.\frac{d}{dt}B_{\psi}(\alpha(t))\right|_{t=0}\\=&\int_0^\infty\int_{\partial\Omega}\psi(y)e^{-l^2|x-y|^2}e^{-\frac{\nu(l|x-y|)}l}l^3\\&\Big[\left(-2l^2v\cdot(x-y)-\nu'(l|x-y|)\frac{v\cdot(x-y)}{|x-y|}\right)
(x-y)\cdot n(x)|(x-y)\cdot n(y)|\\&+v\cdot n(x)|(x-y)\cdot n(y)|+(x-y)\cdot\left.\frac{d}{dt}n(\alpha(t))\right|_{t=0}|(x-y)\cdot n(y)|\\&+(x-y)\cdot n(x) \text{sgn}((x-y)\cdot n(y))(v\cdot n(y))\Big] dA(y)dl.\end{split}
\end{equation}
We note that due to the convexity of $\Omega$, $ \text{sgn}((x-y)\cdot n(y))=-1.$
Let $r_1>0$ be as defined by \eqref{defr} in Proposition \ref{propositionA}.
Since $\Omega$ is a $C^2$, bounded and strictly convex domain, there exists a positive number $r$ such that for each $x \in \partial \Omega$, we have
\begin{equation*}
B_r(x) \cap \partial\Omega \subset GB(x, r_1).
\end{equation*}
 We now break the domain of integration into two parts : $B_r(x) \cap \partial\Omega$ and  $\partial\Omega\setminus B_r(x),$ 
and denote the corresponding integrals as $\nabla^x_vB_\psi^{ s}$ and $\nabla^x_vB_\psi^{ l}(x)$ respectively. First, we estimate  $\nabla^x_vB_\psi^{ l}(x)$. We notice that
\begin{enumerate}
\item[(i)] $v \cdot n(x) =0,$\\
\item[(ii)] $|y-x| \ge r,$ for $y \in \partial \Omega \setminus B_r(x),$\\
\item[(iii)] $\frac{d}{dt} n(\alpha(t))$ is bounded due to smoothness and compactness of $\partial \Omega$,\\
\item[(iv)] $\nu'(l|x-y|) \le C (1+l|x-y|)^{\gamma -1}$, which is uniformly bounded,\\
\item[(v)] as mentioned in the beginning of this section, $\psi$ is bounded since $f \in L^{\infty}_{x, \xi}$ in our context.
\end{enumerate}
Taking (i)-(v) into consideration, we obtain that
\begin{equation}
\label{5first}
\begin{split}
|\nabla^x_vB_\psi^l| &\leq C\int_0^\infty\int_{\partial\Omega\setminus B_r(x)}e^{- r^2 l^2}(l^3+l^5)dA(y)dl\\
&\leq C|\partial \Omega|.
\end{split}
\end{equation}
Secondly, since $B_r(x) \cap \partial\Omega \subset GB(x, r_1),$ we may apply Lemma~\ref{lemmaB} to obtain
\begin{equation}
\label{5second}
\begin{split}|\nabla^x_vB_\psi^s(x)|&\leq C\int_{GB(x,r_1)}\int_0^\infty e^{-l^2|x-y|^2} \\ 
&\quad\quad\quad\big[ l^5|x-y|^5+l^3|x-y|^4+l^3|x-y|^3\big]dldA(y)\\
&\leq C\int_{GB(x,r_1)}\int_0^\infty e^{-z^2}\big[
\frac{z^5}{|x-y|}+z^3+\frac{z^3}{|x-y|}\big]dzdA(y)\\
&\leq C\int_{GB(x,r_1)}(1+\frac1{|x-y|} )dA(y)\\
&\leq C\int_0^{r_1}\int_0^{2\pi}(1+\frac1r)rd\theta dr\leq C. \end{split}
\end{equation}
Notice that here $z=|x-y|l$ and $(r,\theta)$ are polar coordinates for the $T_x(\partial\Omega)$. Combining the estimates \eqref{5first} and \eqref{5second}, the proof of  Lemma \ref{difBpsi} is complete.
\end{proof}

\section{H\"{o}lder continuity of $D_f$}
Recall that in \eqref{defphi3}  we define
\begin{equation}
D_f(x):=2\sqrt{\pi}\int_{\zeta\cdot n > 0}\int_0^{\tau_-(x,\zeta)}e^{-\nu(|\zeta|)s}K(f)(x-\zeta s,\zeta)M^{\frac12}(\zeta)|\zeta\cdot n|dsd\zeta.\end{equation}
In this section,  we shall prove the H\"older continuity of $D_f$.
Let  $\{n(x), e_2, e_3\}$ be an orthonormal basis of $T_x(\partial \Omega)$.  We introduce spherical coordinates so that
\begin{equation}
\zeta=\rho\cos\theta n(x)+\rho \sin \theta \cos\phi e_2+\rho \sin\theta\sin\phi e_3.
\end{equation}
With the further coordinate change : $r=s\rho $, $\hat{\zeta}=\frac{\zeta}{|\zeta|}$, $y=x-r\hat{\zeta}$, we can rewrite $D_f$ as 
\begin{equation}
\begin{split}D_f&=
2{\pi}^{-\frac14}\int_0^\infty\int_0^{\frac{\pi}2}\int_0^{2\pi}\int_0^{|\overline{p(x,\zeta)x}|}e^{-\frac{\nu(\rho)}{\rho}r}K(f)(x-r\hat{\zeta} ,\zeta)e^{-\frac{|\zeta|^2}2}\\
& \qquad \qquad \qquad \qquad \qquad \qquad \times |\zeta\cdot n(x)|\rho\sin\theta dr d\phi d\theta d\rho\\
&=2{\pi}^{-\frac14}\int_0^\infty\int_{\Omega}
e^{-\frac{\nu(\rho)}{\rho}|x-y|}K(f)(y ,\rho\frac{(x-y)}{|x-y|})\\
& \qquad \qquad \qquad \qquad \times\frac{(x-y)\cdot n(x)}{|x-y|}e^{-\frac{\rho^2}2}\frac{\rho^2}{|x-y|^2}dyd\rho.
\end{split}\end{equation}

\begin{lem}
\label{DFDO}
Suppose $x_0$ and $x_1$ are any two points on $\partial \Omega$.  We have  
\begin{equation}
|D_f(x_0)-D_f(x_1)|\leq  C\Vert f\Vert_{L^\infty_{x,\zeta}}|x_0-x_1|\left(1+\big|\ln|x_0-x_1|\big|\right).
\end{equation}
\end{lem}
\begin{proof}\begin{equation}\begin{split}
&|D_f(x_0)-D_f(x_1)|\leq \bigg|2{\pi}^{-\frac14}\int_0^\infty\int_{\Omega}\left[K(f)(y,\rho\frac{(x_0-y)}{|x_0-y|})-[K(f)(y,\rho\frac{(x_1-y)}{|x_1-y|})\right]\\
&\quad\quad \cdot
e^{-\frac{\nu(\rho)}{\rho}|x_0-y|-\frac{\rho^2}2}\rho^2\frac{n(x_0)\cdot(x_0-y)}{|x_0-y|^{3}}dyd\rho\bigg|\\
&\quad+\Bigg|2{\pi}^{-\frac14}\int_0^\infty\int_{\Omega}[K(f)(y,\rho\frac{(x_1-y)}{|x_1-y|})\\
&\quad\quad\cdot\left[
e^{-\frac{\nu(\rho)}{\rho}|x_0-y|-\frac{\rho^2}2}\rho^2\frac{n(x_0)\cdot(x_0-y)}{|x_0-y|^{3}}-e^{-\frac{\nu(\rho)}{\rho}|x_1-y|-\frac{\rho^2}2}\rho^2\frac{n(x_1)\cdot(x_1-y)}{|x_1-y|^{3}}\right]dyd\rho\Bigg|\\
&=:\Delta D_{fK}+\Delta D_{fO}.
\end{split}\end{equation} We first estimate $\Delta D_{fK}.$
We  break the domain of integration into two parts,  $\Omega_1=\Omega\cap B(x_0,2|x_0-x_1|)$ and $\Omega_2:= \Omega\setminus B(x_0,2|x_0-x_1|)$, and denote the corresponding integrals as $\Delta D_{fK}^1$ and $\Delta D_{fK}^2$ respectively. Because of smallness of the domain of integration, by \eqref{defK} and \eqref{estimateK}, one may readily derive that
\begin{equation}
|\Delta D_{fK}^1|\leq C\Vert f\Vert_{L^\infty_{x,\zeta}}|x_0-x_1|.
 \end{equation}
 To estimate $D_{fK}^2$, by employing the Lipschitz continuity of $K(f)$ :
 \begin{equation}
 |K(f)(y,\zeta_1)-K(f)(y,\zeta_2)|\leq C\Vert f \Vert_{L^\infty_{x,\zeta}}|\zeta_1-\zeta_2|, 
 \end{equation}  
 we get 
 \begin{equation}\label{differenceKxy}\begin{split}
& \left|K(f)(y, \rho\frac{(x_0-y)}{|x_0-y|})-K(f)(y, \rho\frac{(x_1-y)}{|x_1-y|})\right|\\
&\leq  C\Vert f \Vert_{L^\infty_{x,\zeta}}\left|\frac{\rho(x_0-y)}{|x_0-y|}-\frac{\rho(x_1-y)}{|x_1-y|}\right|\\
&\leq   C\rho\Vert f \Vert_{L^\infty_{x,\zeta}}\left|\frac{|x_1-y|(x_0-y)-|x_0-y|(x_1-y)}{|x_0-y||x_1-y|}\right|
\\  &\leq C\rho\Vert f \Vert_{L^\infty_{x,\zeta}}\left|\frac{|x_1-y|(x_0-x_1)+(|x_1-y|-|x_0-y|)(x_1-y)}{|x_0-y||x_1-y|}\right|
\\  &\leq C\frac{\rho|x_0-x_1|}{|x_0-y|}\Vert f \Vert_{L^\infty_{x,\zeta}}.\end{split}\end{equation}
Therefore,
\begin{equation}
\begin{split}
&|\Delta D_{fK}^2|\leq C{|x_0-x_1|}\Vert f \Vert_{L^\infty_{x,\zeta}}\int_0^\infty\int_{\Omega_2}e^{-\frac{\nu(\rho)}{\rho}|x_0-y|-\frac{\rho^2}2}\frac{\rho^3}{|x_0-y|^{3}}dyd\rho\\
&\leq C{|x_0-x_1|}\Vert f \Vert_{L^\infty_{x,\zeta}}\int_{0}^{\infty}e^{-\frac{\rho^2}{2}} \rho^3 d \rho \int_{2|x_0-x_1|}^R\frac1rdr\\
& \leq C{|x_0-x_1|}(1+\left| \ln \left| x_0-x_1 \right| \right|)\Vert f \Vert_{L^\infty_{x,\zeta}}.
\end{split}
\end{equation}

Now, we proceed to estimate $\Delta D_{fO}$.  Suppose $0<r<r_1$  as chosen in the proof of Lemma~\ref{difBpsi} so that for every $x \in \partial \Omega$, we have
\begin{equation*}
B_r(x) \cap \partial\Omega \subset GB(x, r_1).
\end{equation*}
{\sc Case 1.} If  $2|x_0-x_1| \ge r,$ we see that
\begin{equation}
\frac{\Delta D_{fO}}{|x_0-x_1|} \le \frac{C\Vert f\Vert_{L^\infty_{x,\zeta}}}{r}.
\end{equation}
{\sc Case 2.} In case of $2|x_0-x_1| < r,$ we split the domain of integration into two parts:  $\Omega_1=\Omega\cap B(x_0,2|x_0-x_1|)$ and $\Omega_2:= \Omega\setminus B(x_0,2|x_0-x_1|)$, and denote the corresponding integrals as $\Delta D_{fO}^1$ and $\Delta D_{fO}^2$ respectively.  Due to smallness of domain of integration, we have
\begin{equation}
|\Delta D_{fO}^1|\leq C\Vert f\Vert_{L^\infty_{x,\zeta}}|x_0-x_1|.
 \end{equation}
 To estimate $\Delta D_{fO}^2$, we let  $x(t)\subset GB(x_0, r_1) \subset \partial \Omega$ be the normal geodesic connecting $x_0$ and $x_1$ with 
 \begin{equation}
 x(0)=x_0,\,\ x(s)=x_1,
 \end{equation}
 where $s$ is the geodesic distance between $x_0$ and $x_1$ on $\partial \Omega.$
 We observe that 
 \begin{align*}
 \frac{d}{dt}& \Big( e^{-\frac{\nu(\rho)}{\rho}|x(t)-y|}\frac{n(x(t))\cdot(x(t)-y)}{|x(t)-y|^{3}} 
 \Big)\\ 
 &= -\frac{\nu(\rho)}{\rho} \frac{x'(t) \cdot (x(t)-y) }{|x(t)-y|}  e^{-\frac{\nu(\rho)}{\rho}|x(t)-y|}\frac{n(x(t))\cdot(x(t)-y)}{|x(t)-y|^{3}} \\
 & \Big(\frac{(\frac{d}{dt} n(x(t))) \cdot (x(t)-y)}{|x(t)-y|^3} -3\frac{\big( n(x(t)) \cdot (x(t)-y) \big) \big( x'(t) \cdot (x(t)-y) \big) }{|x(t)-y|^5} \Big) \\
 & \times e^{-\frac{\nu(\rho)}{\rho}|x(t)-y|}. 
 \end{align*}
 Hence,
 \begin{equation*}
 \Big| \frac{d}{dt} \Big( e^{-\frac{\nu(\rho)}{\rho}|x(t)-y|}\frac{n(x(t))\cdot(x(t)-y)}{|x(t)-y|^{3}} \Big) \Big| \le  e^{-\frac{\nu(\rho)}{\rho}|x(t)-y|}\Big(   \frac{1}{|x(t) -y|^3}
 + (1 +\frac{\nu(\rho)}{\rho} )  \frac{1}{|x(t) -y|^2} \Big).
  \end{equation*}
 By the fundamental theorem of calculus, we derive that
 \begin{equation}\begin{split}
&\Delta D_{fO}^2=\\&\Bigg|2{\pi}^{-\frac14}\int_0^\infty\int_{\Omega_2}K(f)(y,\frac{x_1-y}{|x_1-y|}\rho)\int_0^s\frac{d}{dt}\left(e^{-\frac{\nu(\rho)}{\rho}|{x}(t)-y|-\frac{\rho^2}2}\rho^2\frac{n({x}(t))\cdot({x}(t)-y)}{|{x}(t)-y|^{3}}\right)dtdyd\rho\Bigg|\\
&\leq C\Vert f\Vert_{L^\infty_{x,\zeta}}\Bigg|\int_0^s\int_0^\infty\rho(1+\rho)e^{-\frac{\rho^2}2}\int_{\Omega_2}  (\frac1{|{x}(t)-y|^3} +\frac1{|{x}(t)-y|^2} )dyd\rho dt\Bigg|\\
&\leq  C\Vert f\Vert_{L^\infty_{x,\zeta}}\Bigg|\int_0^s\int_0^\infty\rho(1+\rho)e^{-\frac{\rho^2}2}\int_0^{2\pi}\int_0^{\pi}\int_{|x_0-x_1|}^R(1+\frac1{r})\sin\theta dr d\theta d\phi d\rho dt\Bigg|\\
&\leq  C\Vert f\Vert_{L^\infty_{x,\zeta}} |x_0 - x_1| \left(1+\big|\ln|x_0-x_1|\big|\right),\end{split}\end{equation}
where we have used the  inequality   
$$
s \le \frac{8}{7}|x_0 - x_1|
$$ 
addressed in  Remark~\ref{rem32}. Combining the above estimates of  $\Delta D_{fK}^1$,  $\Delta D_{fK}^2$, $\Delta D_{fO}^1$ and $\Delta D_{fO}^2$, the proof of
Lemma~\ref{DFDO} is complete.
 \end{proof}

\section{{ H\"{o}lder type estimates revisited}\label{HolderRV}}

 From the previous sections, we in fact can already claim interior H\"{o}lder continuity by an argument similar to \cite{RegularChen}. In order to further bootstrap the regularity to differentiability, those estimates  need to be significantly refined.
In this section, we are going to  prepare some estimates for $I$ and $II$  defined in  \eqref{def123}. In the next section, Section \ref{Mixtureenhanced}, we will further improve the Mixture Lemma for stationary solution introduced in  \cite{RegularChen}.

We define\begin{defi}\begin{align}
N(x,\zeta)&:=\frac{|n(p(x,\zeta))\cdot\zeta|}{|\zeta|},\\
d_{x,y}&:= \min\{d_x, d_y\}.
\end{align}
\end{defi}

Our goal in this section is to prove the following lemma.

\begin{lem} \label{HolderIandII}
 Assume that $\Omega$ satisfies the geometric assumptions introduced in Lemma~\ref{lemmaA}. Suppose that  for fixed $0<\epsilon<\frac16$, there exist $a,\  M_0>0$ such that
  \begin{align}
  &|f(X,\zeta)-f(Y,\zeta)|\leq M_0|X-Y|^{1-\epsilon}e^{-a|\zeta|^2},\\
  &|f(X,\zeta)|\leq M_0e^{-a|\zeta|^2},
  \end{align} 
  for all $(X,\zeta),\ (Y,\zeta)\in\Gamma_-$.
  Then, there exists a constant $C$ such that for any $x,\ y \in\Omega$,
  \begin{align}
  &|I(x,\zeta)-I(y,\zeta)|\leq C\frac1{d_{x,y}} |x-y|^{1-\epsilon}e^{-\frac{a}2|\zeta|^2},\label{Iinfest}\\
  & \label{IwithN} |I(x,\zeta)-I(y,\zeta)|\leq C \left(\frac{|x-y|^{1-\epsilon}}{N(x,\zeta)}+\frac{|x-y|}{N(x,\zeta)|\zeta|} \right.\\
 & \nonumber \qquad  \qquad \qquad \qquad \qquad \,\,\, \left. +\frac{|x-y|^{1-\epsilon}}{N(y,\zeta)}+\frac{|x-y|}{N(y,\zeta)|\zeta|} \right)e^{-a|\zeta|^2}
  \\
  &\label{kdIest}\int_{\mathbb{R}^3}|k(\zeta,\zeta')||I(x) -I(y)|d\zeta'\leq \\
&\nonumber  \qquad \qquad \qquad \qquad C(1+d_{x, y}^{-1})^{\frac13}(\big|\ln d_{x,y} \big|+1)|x-y|^{1-\epsilon}.
\end{align}
\end{lem}

We need some observations in geometry  to prove the above lemma. 
\begin{prop}\label{Propdifference}
Let $x$ and $y$ be interior points of $\Omega$. We denote $p(x,\zeta)$ and $p(y,\zeta)$ by $X$ and $Y$ respectively.  Then 
\begin{equation} 
|x-X|\geq \frac{d_x}{N(x,\zeta)} .\label{prop13}
\end{equation}
Further more, if $|x-X|\le |y-Y|,$ then 
\begin{align}
|X-Y|&\leq  \frac1{N(x,\zeta)}|x-y|, \label{prop11}\\
\big||x-X|-|y-Y|\big|&\leq   \frac2{N(x,\zeta)}|x-y|. \label{prop12}
\end{align}
\end{prop}

\begin{proof}
Let us first prove \eqref{prop13}.  Let $F$ be the projection of $x$ on the tangent plane  $T_{x}(\partial \Omega).$ Because of convexity, $\overline{xF}$ intersects $\partial\Omega$ at one point $F'$. 
Then,\begin{equation}
|x-X|N(x,\zeta)=|x-F|\geq |x-F'|\geq d_x,\end{equation}
which implies \eqref{prop13}.  

When $(x-y)  \parallelsum \zeta$, \eqref{prop11} and \eqref{prop12} are trivial. If not, 
we let \begin{align}
e_1&:=\frac{\zeta}{|\zeta|},\\e_3&:=\frac{e_1\times(y-x)}{|e_1\times(y-x)|},\\e_2&:=e_3\times e_1.
\end{align}
Also, we denote
\begin{align}
&n_1= n(X)\cdot e_1,\\&n_2= n(X)\cdot e_2,\\&n_3= n(X)\cdot e_3\\ &n'=n_1e_1+n_2e_2.
\end{align}
Notice that $n_1^2+n_2^2+n_3^2=1$  and $N(x,\zeta)=|n_1|$. 
Let $E$ be the plane containing  $x$, $y$, $X$ and $Y$ and  $\Gamma^*=\partial \Omega \cap E.$   We are going to discuss plane geometry on the plane $E.$
Since $|X-x| \le |Y-y|$, the point $y^* : = y + X-x$ lies on the line segment $\overline{yY}.$ If $y^* = Y$, it is obvious that \eqref{prop11} and \eqref{prop12} hold true. In what follows, we assume $y^* \ne Y.$ Due to the convexity of $\Omega$, the tangent line of $\Gamma^*$ passing $X$ would intersect the half line $\overrightarrow{\rm yY}$ at a single point $Y^*.$ For the sake of convenience, we define
\begin{equation*}
\theta_1 = \angle XYy^*, \quad \theta_2 = \angle XY^*Y, \quad \theta_3 = \angle Yy^*X.
\end{equation*}
By the law of sines, we see that
\begin{align*}
& \frac{\overline{Xy^*}}{\sin \theta_1} = \frac{\overline{XY}}{\sin \theta_3}  \qquad \text{ and } 
&  \frac{\overline{Xy^*}}{\sin \theta_2} = \frac{\overline{XY^*}}{\sin \theta_3}.
\end{align*} 
In case $\theta_1 \ge \frac{\pi}{2}$ it is obvious that
\begin{equation}
\label{xy1}
\overline{XY} < \overline{Xy^*} = \overline{xy}.
\end{equation}
In case $\theta_1 < \frac{\pi}{2},$ by monotonicity of sine function on the interval $[0, \frac{\pi}{2}]$ and the fact  $\theta_1 > \theta_2$  , we see
\begin{equation}
\label{xy2}
\overline{XY} =\frac{\sin \theta_2 }{\sin \theta_1} \overline{XY^*} < \overline{XY^*}  <\frac{1}{\sin \theta_2} \overline{Xy^*}=\frac{1}{\sin \theta_2} \overline{xy}.\end{equation}
 On the other hand, one may readily see that
 $$\sin \theta_2 = \left| e_1 \cdot \frac{n'}{|n'|} \right| = \left| \frac{n_1}{\sqrt{n_1^2 + n_2^2}}  \right|.$$
Summing up of \eqref{xy1} and \eqref{xy2},  in any case, we obtain that
\begin{equation}
|X-Y| \le \frac{1}{\sin \theta_2}|x-y| \le \frac{1}{|n_1|}|x-y| = \frac{1}{N(x, \zeta)}|x-y|. 
\end{equation}
Finally, 
\begin{equation}
\big||x-X|-|y-Y|\big| \le | X-Y| + |x-y| \le  \frac{2}{N(x, \zeta)}|x-y|.
\end{equation}
This completes the proof of Proposition~\ref{Propdifference}.
\end{proof}

Next, we shall prove the following proposition which has been mentioned  in \cite{RegularChen}.

\begin{prop}\label{pointtobd}
Let $\Omega$ be a $C^1$ bounded convex domain  in $\mathbb{R}^3$. Suppose $x\in \Omega$, $X=p(x,\zeta)\in\partial\Omega$,  and $z\in\overline{xX}$.
Then, 
 \begin{equation}
d_z\geq \frac{d_x}{R}|z-X|, \end{equation}
where $R$ is the diameter of $\Omega$.
\end{prop}
 \begin{proof} 
 We denote a point on $\partial \Omega$ that realizes $d_z$ by $Z$. Let $L_x$ be the plane passing $x$ perpendicular to $\zeta$. Let $L_Z$ be the plane passing $Z$, $z$, and $x$. We denote the intersection point of $L_x$, $L_Z$, and $\partial \Omega$ on the same side with $Z$ on $L_Z$ to $\overline{Xx}$ by $A$. Let $\theta_1=\angle{ZXz}$, $\theta_2=\angle{XZz}$, and $\theta_1'=\angle{AXz}$.
  Due to the convexity of $\Omega$, we have \begin{equation}
  1\geq \sin\theta_1 \geq \sin\theta'_1= \frac{|A-x|}{|A-X|}\geq \frac{d_x}{R},
  \end{equation}
  where $R$ is the diameter of $\Omega$. 
  By the law of sines, \begin{equation}
  d_z\geq d_z \sin\theta_2=|X-z|\sin\theta_1 \ge |X-z|\frac{d_x}{R}.
  \end{equation}
  This concludes the proposition.
  \end{proof}

Now, we are ready to prove Lemma \ref{HolderIandII}.
\begin{proof}  
Without loss of generality, we may assume $|X-x|\leq |Y-y|.$ Hence,
\begin{equation}
\label{Ixy}
\begin{split}
&|I(x,\zeta)-I(y,\zeta)|=|f(X,\zeta)e^{-\nu(|\zeta|)\frac{|X-x|}{|\zeta|}}-f(Y,\zeta)e^{-\nu(|\zeta|)\frac{|Y-y|}{|\zeta|}}|\\
&\leq |f(X,\zeta)-f(Y,\zeta)|e^{-\nu(|\zeta|)\frac{|X-x|}{|\zeta|}}+|f(Y,\zeta)|\big|e^{-\nu(|\zeta|)\frac{|X-x|}{|\zeta|}}-e^{-\nu(|\zeta|)\frac{|Y-y|}{|\zeta|}}\big|\\
&\leq Ce^{-a|\zeta|^2}\left((\frac{|x-y|}{N(x,\zeta)})^{1-\epsilon} +(\frac{|x-y|}{N(x,\zeta)|\zeta|}) \right)e^{-\nu_0\frac{d_x}{N(x,\zeta)|\zeta|}}\\
&\leq  Ce^{-a|\zeta|^2}\left((\frac{|\zeta||x-y|}{d_x})^{1-\epsilon}+(\frac{|x-y|}{d_x})\right)\\
&\leq  Cd_x^{-1}|x-y|^{1-\epsilon}e^{-\frac{a}{2}|\zeta|^2}.\end{split}\end{equation}
Notice that we have used the mean value theorem and Proposition \ref{Propdifference} in the above estimate.
We observe that the third line of \eqref{Ixy} gives
\begin{equation}
\label{Ixy2}
|I(x,\zeta)-I(y,\zeta)|\leq C \left(\frac{|x-y|^{1-\epsilon}}{N(x,\zeta)}+\frac{|x-y|}{N(x,\zeta)|\zeta|}\right)e^{-a|\zeta|^2},
\end{equation}
Due to the symmetry of $x$ and $y$, by \eqref{Ixy} and \eqref{Ixy2}, we obtain \eqref{Iinfest} and \eqref{IwithN}.
To prove \eqref{kdIest}, we first divide the domain of integration into two:
 \begin{align}B_0&:=\big\{\zeta'\in\mathbb{R}^3\big||\zeta-\zeta'|< d_{x,y}^{\frac13}\big\},\\  B_0^c&:=\mathbb{R}^3\setminus B_0.\end{align} 

We denote the corresponding integrals by $K\Delta I_1$ and $K\Delta I_2$ respectively. Using \eqref{Iinfest}, we have
\begin{equation}\begin{split}|K\Delta I_1|&\leq C\int_{B_0}\frac{|x-y|^{1-\epsilon}}{|\zeta-\zeta'|d_{x,y}}e^{-\frac a2|\zeta'|^2}d\zeta'\\
& \leq C\frac{|x-y|^{1-\epsilon}}{d_{x,y}}\int_0^{d_{x,y}^{\frac13}}\int_0^\pi\int_0^{2\pi}r\sin\theta d\phi d\theta dr\\
&\leq C d_{x,y}^{-\frac13}|x-y|^{1-\epsilon},\end{split}\end{equation}
where we used the spherical coordinates centered at $\zeta$ in the above inequality. 

 Using \eqref{IwithN} and changing variable similar to Section \ref{SecDiffBpsi}, we have
\begin{equation}\begin{split}
&|K\Delta I_2|\leq C\int_{ B_0^c} \frac{e^{-a |\zeta'^2|}}{d_{x,y}^{\frac13}}\left(\frac{|x-y|^{1-\epsilon}}{N(x,\zeta')}+\frac{|x-y|}{N(x,\zeta')|\zeta'|}+\frac{|x-y|^{1-\epsilon}}{N(y,\zeta')}+\frac{|x-y|}{N(y,\zeta')|\zeta'|}\right)d\zeta'\\
& \leq C\int_0^\infty\int_{\partial\Omega}\frac{e^{-al^2|x-z|^2}}{d_{x,y}^{\frac13}}
\left(\frac{|x-z||x-y|^{1-\epsilon}}{|(x-z)\cdot n(z)|}+\frac{|x-y|}{|(x-z)\cdot n(z)|l}\right)l^2|(x-z)\cdot n(z)|dA(z)dl\\
&\quad +C\int_0^\infty\int_{\partial\Omega}\frac{e^{-al^2|y-z|^2}}{d_{x,y}^{\frac13}}
\left(\frac{|y-z||x-y|^{1-\epsilon}}{|(y-z)\cdot n(z)|}+\frac{|x-y|}{|(y-z)\cdot n(z)|l}\right)l^2|(y-z)\cdot n(z)|dA(z)dl\\
&\leq \frac{C}{d_{x,y}^{\frac13}} \int_{\partial\Omega}\int_0^\infty e^{-al^2|x-z|^2}\left(l^2|x-z||x-y|^{1-\epsilon}+l|x-y|)\right)dldA(z)\\
& \quad+\frac{C}{d_{x,y}^{\frac13}} \int_{\partial\Omega}\int_0^\infty e^{-al^2|y-z|^2}\left(l^2|y-z||x-y|^{1-\epsilon}+l|x-y|)\right)dldA(z)\\
&\leq \frac{C}{d_{x,y}^{\frac13}} \int_{\partial\Omega}\int_0^\infty e^{-as^2}\left(s^2|x-y|^{1-\epsilon}+s|x-y|)\right)ds\left(\frac1{|x-z|^2}+\frac1{|y-z|^2}\right)dA(z) \\
&\leq \frac{C}{d_{x,y}^{\frac13}}|x-y|^{1-\epsilon} \int_{\partial\Omega}\left(\frac1{|x-z|^2}+\frac1{|y-z|^2}\right)dA(z)\\
&\leq \frac{C}{d_{x,y}^{\frac13}}|x-y|^{1-\epsilon}(\big|\ln d_x\big|+\big|\ln d_y\big|+1)\leq \frac{C}{d_{x,y}^{\frac13}}|x-y|^{1-\epsilon}(\big|\ln d_{x,y}\big|+1).
\end{split}\end{equation}

Notice that, in the above estimates,  we changed the variable $s=|x-z|l$, $s=|y-z|l$ and applied Lemma \ref{lemmaA}.  The proof of  Lemma \ref{HolderIandII} is complete.

\end{proof}

\section{Regularity due to Mixing \label{Mixtureenhanced}}

We shall elaborate the smoothing effect due to the combination of collision and transport in this section.  In the following proposition, we improve the estimate in \cite{RegularChen} from H\"{o}lder continuity with order $\frac12-$ to almost Lipschitz continuous.
\begin{prop} \label{HolderGln}
Suppose $f\in L^\infty_{x,\zeta}$ is a solution to the stationary linearized Boltzmann equation. Then, for all $x_0, x_1 \in \Omega$ and $\zeta \in \mathbb{R}^3$,
\begin{equation}
|G(x_0,\zeta)-G(x_1,\zeta)|\leq C\Vert f\Vert_{L^\infty_{x,\zeta}}|x_0-x_1|(1+\big|\ln|x_0-x_1|\big|).
\end{equation}
\end{prop}

\begin{proof}As in \cite{RegularChen}, 
We observe that
 \begin{equation}
\begin{split}
&\left|G(x_0,\zeta)-G(x_1,\zeta)\right|\leq
\bigg|\int_0^\infty\int_{\Omega} k \left( \zeta,\rho\frac{(x_0-y)}{|x_0-y|} \right)\frac{\rho e^{-\nu(\rho)\frac{|x_0-y|}{\rho}}}{|x_0-y|^2}\\&\quad\cdot\left[K(f) \left( y, \rho\frac{(x_0-y)}{|x_0-y|} \right)-K(f) \left( y, \rho\frac{(x_1-y)}{|x_1-y|} \right) \right] dyd\rho\bigg|\\ &+\left|\int_0^\infty\int_{\Omega} K(f) \left(y, \rho\frac{(x_1-y)}{|x_1-y|} \right) \right.\\&\quad\cdot \left. \left[k \left(\zeta,\rho\frac{(x_0-y)}{|x_0-y|} \right)\frac{\rho e^{-\nu(\rho)\frac{|x_0-y|}{\rho}}}{|x_0-y|^2}-k \left(\zeta,\rho\frac{(x_1-y)}{|x_1-y|} \right)\frac{\rho e^{-\nu(\rho)\frac{|x_1-y|}{\rho}}}{|x_1-y|^2}\right]dyd\rho\right|\\&=:G_K+G_O.
\end{split}
\end{equation} The estimate for $G_O$ has already been done in \cite{RegularChen}.  We only need to  estimate  $G_K$.  We  break the domain of integration into two,  $\Omega_1=\Omega\cap B(x_0,2|x_0-x_1|)$ and $\Omega_2:= \Omega\setminus B(x_0,2|x_0-x_1|)$, and name the corresponding integrals $G_K^1$ and $G_K^2$ respectively. Because of smallness of the domain of integration, we have
\begin{equation}
|G_K^1|\leq C\Vert f\Vert_{L^\infty_{x,\zeta}}|x_0-x_1|.
 \end{equation}
 To deal with $G_K^2$, we need to use the Lipschitz continuity of $K(f)$ \eqref{differenceKxy},
 \begin{equation}\begin{split}\notag
& \left|K(f) \left(y, \rho\frac{(x_0-y)}{|x_0-y|} \right)-K(f) \left(y, \rho\frac{(x_1-y)}{|x_1-y|} \right)\right|  \leq C\rho\frac{|x_0-x_1|}{|x_0-y|}\Vert f \Vert_{L^\infty_{x,\zeta}}.\end{split}\end{equation}
Therefore, by taking the coordinate change $$y-x_0=(r \cos\theta, r \sin \theta\cos\phi,r\sin \theta\sin \phi),$$
we see that
\begin{equation}
\begin{split}
&|G_K^2|\leq C{|x_0-x_1|}\Vert f \Vert_{L^\infty_{x,\zeta}}\int_0^\infty\int_{\Omega_2} \left| k\left(\zeta,\rho\frac{(x_0-y)}{|x_0-y|} \right) \right| \frac{\rho^2 e^{-\nu(\rho)\frac{|x_0-y|}{\rho}}}{|x_0-y|^3}dyd\rho\\&\leq C{|x_0-x_1|}\Vert f \Vert_{L^\infty_{x,\zeta}}\int_{\mathbb{R}^3}|k(\zeta,\zeta')|\int_{2|x_0-x_1|}^R\frac1rdrd\zeta'\\& \leq C{|x_0-x_1|}\left( 1+\big|\ln|x_0-x_1|\big| \right)\Vert f \Vert_{L^\infty_{x,\zeta}}.
\end{split}
\end{equation}

\end{proof}
With the above proposition, we can prove the following estimate.
\begin{prop}\label{HolderIII} Suppose $\Omega$  satisfies the same geometric assumption as mentioned in Lemma~\ref{lemmaA}. Then, the following inequality holds 
\begin{equation}
\label{dxy}
|III(x,\zeta)-III(y,\zeta)|\leq C\Vert f\Vert_{L^\infty_{x,\zeta}}(1+d_{x,y}^{-1})|x-y|^{1-\epsilon}.
\end{equation}
\end{prop}
\begin{proof} Let $X=p(x,\zeta)$ and $Y=p(y,\zeta)$. We will demonstrate the proof for  the case when $|X-x|\leq |Y-y|$. The other case can be proved in the same fashion. 
Noting that  
$$
   \tau_{-}(x,\zeta) = \frac{|X-x|}{|\zeta|}  \,\, \text{ and }    \tau_{-}(y,\zeta) = \frac{|Y-y|}{|\zeta|},   $$
  we have
\begin{equation}
\begin{split}
|III(x,\zeta)-III(y,\zeta)|\leq&\int_0^{\frac{|X-x|}{|\zeta|}}e^{-\nu(|\zeta|)s}|G(x-s\zeta,\zeta)-G(y-s\zeta,\zeta)|ds\\&\quad+\int_{\frac{|X-x|}{|\zeta|}}^{\frac{|Y-y|}{|\zeta|}}e^{-\nu(|\zeta|)s}|G(y-s\zeta,\zeta)|ds\\&=:\Delta_1+\Delta_2.
\end{split}
\end{equation}
By applying Proposition \ref{HolderGln}, we have
\begin{equation}
\Delta_1\leq C \Vert f \Vert_\infty|x-y|^{1-\epsilon}.\end{equation}
On the other hand
\begin{equation}
\begin{split}\Delta_2&\leq C\Vert f \Vert_\infty\int_{\frac{|X-x|}{|\zeta|}}^{\frac{|Y-y|}{|\zeta|}}e^{-\nu(|\zeta|)s}ds\\
&\leq C\Vert f \Vert_\infty\left|e^{\frac{-\nu(|\zeta|)}{|\zeta|}|X-x|} -e^{-\frac{\nu(|\zeta|)}{|\zeta|}|Y-y|}\right|\\
&\le  C\Vert f \Vert_\infty e^{\frac{-\nu(|\zeta|)}{|\zeta|}|X-x|}\left|{\frac{\nu(|\zeta|)}{|\zeta|}(|Y-y|-|X-x|)}\right|,
\end{split}\end{equation}
where we have applied the mean value theorem and used the assumption $|X-x| \le |Y-y|.$ We the apply  Proposition \ref{Propdifference} and get 
\begin{equation}
\label{dx}
\begin{split}
\Delta_2&\leq C\Vert f \Vert_\infty e^{-\frac{\nu(|\zeta|)}{|\zeta|} \frac{d_x}{N(x,\zeta)}}\frac{\nu(|\zeta|)}{|\zeta|} \frac{|x-y|}{N(x,\zeta)}\\&\leq C\Vert f \Vert_\infty \frac{|x-y|}{d_x} .\end{split}
\end{equation}
Similarly, in case $|Y-y| \le |X-x|$, we can conclude that 
\begin{equation}
\label{dy}
\Delta_2\leq C\Vert f \Vert_\infty \frac{|x-y|}{d_y} .
\end{equation}
By \eqref{dx} and \eqref{dy}, we obtain \eqref{dxy} and the proof of Proposition~\ref{HolderIII} is complete.

\end{proof}

\section{Behavior near the boundary\label{LocalHatBoundary}.}
In this section, we  investigate the behavior of $f$ near the boundary.  This is a preparation for proving the differentiability of $D_f$. 

\begin{lem}\label{HolderBoundary}
 Let $\Omega$ be the domain introduced in Lemma~\ref{lemmaA}. Assume  $f\in L^\infty_{x,\zeta}$ is a solution to the stationary linearized Boltzmann equation such that for fixed $0<\epsilon<\frac16$ 
,  there exist $a,\  M_0>0$ such that
  \begin{align}
 \label{c91} &|f(X,\zeta)-f(Y,\zeta)|\leq M_0|X-Y|^{1-\epsilon}e^{-a|\zeta|^2},\\
 \label{c92} &|f(X,\zeta)|\leq M_0e^{-a|\zeta|^2}
  \end{align} 
  for all $(X,\zeta),\ (Y,\zeta)\in\Gamma_-$.
  Then,  for $x\in \partial \Omega$ and $y\in \Omega$, 
  \begin{align}\label{fHolderBndy}
 | f(x,\zeta)-f(y,\zeta)|\leq C \left( 1+\frac1{|\zeta|} \right)|x-y|^{\frac12(1-\epsilon)}.
   \end{align} 
  \end{lem}
The above estimate  not only gives a description on how singular $f$ can be when $\zeta$ is small near the boundary but also plays an important role in proving the differentiability of boundary flux and therefore the solution itself, which will be elaborated in next section. 

\begin{proof}Without loss of generality, we may only consider the case $|x-y|<\min(\frac15 R_0,1)$, where $R_0$ is introduced in Lemma \ref{Atboundarysqrt}. This case gives the unique projection of $y$ on $\partial\Omega$ denoted by $Y_\perp,$ i.e., $|y-Y_\perp|=d_y.$

{\sc Step 1.}
We start with the case $\zeta\cdot n(Y_\perp)<0$ and denote $Y_0=p(y,\zeta)$.
We observe that \begin{equation}\begin{split}
\label{e94}
|x-Y_0|\leq |x-y|+|y-Y_0|\leq C|x-y|^{\frac12},
\end{split}\end{equation}
where we have applied Lemma~\ref{Atboundarysqrt} by taking $x = Y_0 \in \partial \Omega$ in \eqref{sqrtgeq} and used the fact $|x-y| \le 1.$

Noting that $x , Y_0 \in \partial \Omega,$ by \eqref{c91} and the mean value theorem,  we see that 
\begin{equation}
\label{e95}
\begin{split}
&|f(x,\zeta)-f(y,\zeta)|\\
&=\Big|f(x,\zeta)-f(Y_0,\zeta)e^{-\nu(|\zeta|)\frac{|y-Y_0|}{|\zeta|}}\\
&\,\, -\int_0^{\frac{|y-Y_0|}{|\zeta|}} e^{-\nu(|\zeta|)t}K(f)(y-\zeta t,\zeta)dt\Big|\\
&\leq \left|f(x,\zeta)-f(Y_0,\zeta)\right|+|f(Y_0,\zeta)|\left|1-e^{-\nu(|\zeta|)\frac{|y-Y_0|}{|\zeta|}} \right|\\&\quad+\left|\int_0^{\frac{|y-Y_0|}{|\zeta|}} e^{-\nu(|\zeta|)t}K(f)(y-\zeta t,\zeta)dt\right|\\&\leq M_0|x-Y_0|^{1-\epsilon}+C\Vert f\Vert_{L^\infty_{x,\zeta}}\frac{|y-Y_0|}{|\zeta|}\\&\leq C \left( 1+\frac1{|\zeta|} \right)|x-y|^{\frac12(1-\epsilon)}.  \end{split}\end{equation}

{\sc Step 2.}
In case  $\zeta\cdot n(Y_\perp)\geq0$, we define
 \begin{align}
& D_0=\{b\in \Omega| d(b,\partial\Omega)\geq 4|x-y| \},\\
& D_1=\{b\in \Omega| d(b,\partial\Omega)\geq 5|x-y| \}.
\end{align}
Notice that $D_0$ is also a smooth convex domain.
We denote $Y_0=p(y,\zeta)$ and $X_0=p(x,\zeta)$. If either $\overline{xX_0}$ or $\overline{yY_0}$ intersects $\partial D_0 $
 less than twice, we can conclude that \begin{equation} |x-X_0|\leq C|x-y|^\frac12,\ |y-Y_0|\leq C|x-y|^\frac12.\end{equation}
Therefore, we can show \eqref{fHolderBndy} holds  as we proved in Step 1.
Namely, if one of $\overline{xX_0}$ or $\overline{yY_0}$ intersects $\partial D_0 $
 less than twice, then both of $\overline{xX_0}$ and $\overline{yY_0}$ intersect $\partial D_1$
 less than twice. Hence, by the proof of Lemma~\ref{Atboundarysqrt}, we have 
 \begin{align}
 & |x-X_0|\leq C_1' (5|x-y|)^{\frac{1}{2}} \text{ and } \\
  & |y-Y_0| \leq C_1' (5|x-y|)^{\frac{1}{2}}.
   \end{align}
 This implies \eqref{e94}, and  hence \eqref{fHolderBndy} holds.
   
Next, we shall discuss the case that both line segments $\overline{xX_0}$ and $\overline{yY_0}$  intersect $\partial D_0$ twice.
Let $X_1$ and $X_2$ be intersection points of $\overline{xX_0}$ and  $\partial D_0$ in the order $x,$ $X_1,$ $X_2$, and $X_0$ on $\overline{xX_0}$. Also let $Y_1$ and $Y_2$ be intersection points of $\overline{yY_0}$ and $\partial D_0$ respectively. Let $X_{1\perp}$ be the unique projection of $X_1$ on $\partial \Omega$. Since $|X_1-X_{1\perp}|$ realizes the distance between $\partial \Omega$ and $\partial D_0$,   we have
 \begin{equation}n(X_{1\perp})=n(X_1)=\frac{(X_{1\perp}-X_1)}{|X_{1\perp}-X_1|}.\end{equation} Taking the tangent plane of $\partial D_0$ at $X_1$ into consideration and applying
Lemma~\ref{Atboundarysqrt}, we have
\begin{equation}
|X_1-x|\leq 4 C_1'|x-y|^{\frac12},\label{smallsqrt2}\\
\end{equation}
By the same fashion, we have
\begin{align}
|X_0-X_2|&\leq 4 C_1'|x-y|^{\frac12}.\label{smallsqrt1}\\
 |Y_0-Y_2|&\leq 4 C_1'|x-y|^{\frac12},\label{smallsqrt3}\\
  |y-Y_1|&\leq 4 C_1'|x-y|^{\frac12}.\label{smallsqrt4}\end{align}
In what follows, we shall discuss  estimates for the differences of $I$, $II$,  and $III$ defined on \eqref{def123}.  

Regarding the estimate for $I$, we start with the following claim.
\begin{align}
&|X_0-Y_0|\leq C|x-y|^{\frac12},\label{sqrtB1}\\&\big||X_0-x|-|Y_0-y|\big|\leq C|x-y|^{\frac12}\label{sqrtB2}.
\end{align}  
Without loss of generality, we may assume $|Y_0-y|<|X_0-x|$. Let $X'=Y_0+(x-y)$. Notice that since $|x-y| < 4|x-y|$ and $|Y_0-y|<|X_0-x|,$ we see that $X'\in \overline{X_2X_0}$.  The unique projection of $X'$ on $\partial\Omega$ is denoted by $X'_\perp$. Notice that by the convexity of the set  $\{b|d(b,\partial\Omega)\geq d_{X'}\}$, we have $n(X'_\perp)\cdot (X_0-X')\geq0$. Therefore, we can apply Lemma \ref{Atboundarysqrt} and prove the claim. 
Now, we have
\begin{equation}\begin{split}
&|I(x,\zeta)-I(y,\zeta)| = \left| f(X_0,\zeta)e^{-\nu(|\zeta|)\frac{|X_0-x|}{|\zeta|}}-f(Y_0,\zeta)e^{-\nu(|\zeta|)\frac{|Y_0-y|}{|\zeta|}} \right|\\&\leq \left| f(X_0,\zeta)e^{-\nu(|\zeta|)\frac{|X_0-x|}{|\zeta|}} \right| \left|1-e^{-\nu(|\zeta|)\frac{|Y_0-y|-|X_0-x|}{|\zeta|}}\right|\\&\quad+|f(X_0,\zeta)-f(Y_0,\zeta)|e^{-\nu(|\zeta|)\frac{|Y_0-y|}{|\zeta|}}\\&\leq C \left( 1+\frac1{|\zeta|} \right)|x-y|^{\frac{1}{2}(1-\epsilon)}.
\end{split}
\end{equation}

Similarly, for $III$, without loss of generality we may assume $|Y_0-y|\leq |X_0-x|$. By applying Proposition~\ref{HolderGln}, we see that
\begin{equation}
\begin{split}
&|III(x,\zeta)-III(y,\zeta)|\leq \left|\int_0^{\frac{|Y_0-y|}{|\zeta|}}e^{-\nu(|\zeta|)s}|G(x-s\zeta ,\zeta)ds-G(y-s\zeta ,\zeta)|ds\right|\\&\quad+\left|\int_{\frac{|Y-y|}{|\zeta|}}^{\frac{|X_0-x|}{|\zeta|}}e^{-\nu(|\zeta|)s}G(x-s\zeta ,\zeta)ds\right|\\&\leq 
C\Vert f\Vert_{L^\infty_{x,\zeta}}\left[\int_0^{\frac{|Y_0-y|}{|\zeta|}} e^{-\nu(|\zeta|)s} |x-y|^{1-\epsilon}ds+\frac{\big||X_0-x|-|Y_0-y|\big|}{|\zeta|}\right]\\& \leq C \left( 1+\frac1{|\zeta|} \right)|x-y|^{\frac12}.
\end{split}
\end{equation}

To estimate the difference of  $II$, thanks to $X'=Y_0+(x-y) \in  \overline{X_2X_0} $,
we may express\begin{equation}\begin{split}
&|II(x,\zeta)-II(y,\zeta)|=\bigg|\int_0^{\tau_-(x,\zeta)}e^{-\nu(|\zeta|)t}\int_{\mathbb{R}^3}k(\zeta,\zeta')I(x-t \zeta,\zeta') d\zeta'dt
\\&-\int_0^{\tau_-(y,\zeta)}e^{-\nu(|\zeta|)t}\int_{\mathbb{R}^3}k(\zeta,\zeta')I(y-t \zeta,\zeta') d\zeta'dt\bigg|\\&\leq \int_0^{\frac{|x-X_1|}{|\zeta|}}e^{-\nu(|\zeta|)t}\int_{\mathbb{R}^3}|k(\zeta,\zeta')||I(x-t \zeta,\zeta')|d\zeta'dt\\&\quad+\int_0^{\frac{|x-X_1|}{|\zeta|}}e^{-\nu(|\zeta|)t}\int_{\mathbb{R}^3}|k(\zeta,\zeta')||I(y-t \zeta,\zeta')|d\zeta'dt
\\&\quad+
\int_{\frac{|x-X_1|}{|\zeta|}}^{\frac{|x-X_2|}{|\zeta|}}e^{-\nu(|\zeta|)t}\int_{\mathbb{R}^3}|k(\zeta,\zeta')||I(x-t \zeta,\zeta') -I(y-t\zeta,\zeta')|d\zeta'dt\\&\quad+\int_{\frac{|x-X_2|}{|\zeta|}}^{\frac{|x-X_0|}{|\zeta|}}e^{-\nu(|\zeta|)t}\int_{\mathbb{R}^3}|k(\zeta,\zeta')||I(x-t \zeta,\zeta')|d\zeta'dt\\&\quad+\int_{\frac{|x-X_2|}{|\zeta|}}^{\frac{|y-Y_0|}{|\zeta|}}e^{-\nu(|\zeta|)t}\int_{\mathbb{R}^3}|k(\zeta,\zeta')||I(y-t \zeta,\zeta')|d\zeta'dt\\&=:DII_1+DII_2+DII_3+DII_4+DII_5.\end{split}\end{equation}
Notice that \begin{equation}
DII_1+DII_2+DII_4+DII_5\leq C\frac{|x-y|^{\frac12}}{|\zeta|}
\end{equation}
because of smallness of the domain of integration, \eqref{smallsqrt1}-\eqref{smallsqrt4}.
We are now focus on $DII_3$. 
Let $\hat{\zeta}=\frac{\zeta}{|\zeta|}$. 
We can rewrite
\begin{equation}
DII_3=\int_{|x-X_1|}^{{|x-X_2|}}\frac{1}{|\zeta|}e^{-\frac{\nu(|\zeta|)}{|\zeta|}r}\int_{\mathbb{R}^3}|k(\zeta,\zeta')||I(x-r\hat{\zeta},\zeta') -I(y-r\hat{\zeta},\zeta')|d\zeta'dr
\end{equation}

 Let $z(r)=x-r\hat{\zeta}$ and  $z'(r)=y-r\hat{\zeta}$.  Notice that since $X'=Y_0+(x-y)$ lies on $\overline{X_2X_0}$, we see that  $z'(r)=y-r\hat{\zeta} \in \overline{y Y_0}$  for $ |x-X_1|\leq r\leq |x-X_2|$ and

\begin{equation}
d_{z'}\geq \frac12d_z+(\frac12d_z-|z-z'|)\geq \frac12d_z \geq 2|x-y|.
\end{equation}

Inferring from \eqref{kdIest}, we have\begin{equation}
\int_{\mathbb{R}^3}|k(\zeta,\zeta')||I(z(r),\zeta) -I(z'(r),\zeta)|d\zeta'\leq C \left( 1+d_{z(r)}^{-1} \right)^{(\frac13+\epsilon')}|x-y|^{1-\epsilon},
\end{equation}
for $0<\epsilon'<1$.
Let  $Z$ be the midpoint of $X_1$ and $X_2$. Notice that $d_Z\geq 4|x-y|$. 

Therefore,
\begin{equation}\begin{split}
DII_3&\leq C \int_{|X_1-x|}^{|x-Z|}\frac{|x-y|^{1-\epsilon}}{|\zeta|}(1+|x-y|r)^{-(\frac13+\epsilon')}\,dr\\&+C\int_{|X_2-X_0|}^{|X_0-Z|}\frac{|x-y|^{1-\epsilon}}{|\zeta|}(1+|x-y|r)^{-(\frac13+\epsilon')}\,dr\\&\leq C\frac{|x-y|^{\frac23-\epsilon-\epsilon'}}{|\zeta|}\end{split}
\end{equation}
Choosing small enough $\epsilon'$, we can conclude the lemma.

\end{proof}
\section{Differentiability of boundary flux \label{SDiffDf}}
Let $g(t)$ be a normal geodesic on $\partial \Omega$ such that \begin{align}
&g(0)=x,\\&g'(0)=v,\end{align}where $v\in T_x (\partial\Omega)$

Recall  that by definition
\begin{equation}
\nabla^x_{v}D_f(x)=\left.\frac{d}{d t}D_f(g(t))\right|_{t=0},
\end{equation}
where
\begin{equation}
\label{Dfg}
\begin{split}
&D_f(g(t))=2\pi^{-\frac14}\int_0^\infty\int_{\Omega}\int_{\mathbb{R}^3}e^{-\frac{\nu(\rho)}\rho|g(t)-y|}k \left( \rho\frac{g(t)-y}{|g(t)-y|},\zeta' \right)\frac{(g(t)-y)\cdot n(g(t))}{|g(t)-y|^3}\\&\quad\quad\quad\quad\quad\quad e^{-\frac{\rho^2}2}\rho^2f(y,\zeta')d\zeta'dyd\rho.
\end{split}\end{equation} 
We will devote this section to prove the following lemma. 

\begin{lem} \label{Dfdiff}
There exists a constant $C$ such that 
 for any  $v\in T_x\partial \Omega$ with $|v|=1$, we have
\begin{equation}
|\nabla^x_{v}D_f(x)|\leq C.\end{equation}
\end{lem} If we differentiate the formula \eqref{Dfg} directly, 
we will obtain a singularity of  $|x-y|^{-3}$, which is barely not integrable in $\Omega$. However, by subtracting and adding $f(x,\zeta')$,  we can use the  local H\"{o}lder continuity in Lemma \ref{HolderBoundary} to make it integrable. For the additional term we introduced, we observe that we can rewrite all the terms except the term with derivative of $n(x)$ as a divergence form and apply the divergence theorem.  More precisely, we   can write

\begin{equation}\begin{split}
&2^{-1} \pi^{\frac14} \nabla^x_v D_f(x)=\int_0^\infty\int_{\Omega}\int_{\mathbb{R}^3} \nabla^x_v\left(e^{-\frac{\nu(\rho)}{\rho}|x-y|}k \left( \rho\frac{x-y}{|x-y|},\zeta' \right)\frac{(x-y)\cdot n(x)}{|x-y|^3}\right)\\
&\quad\cdot\left[f(y,\zeta')-f(x,\zeta')\right]e^{-\frac{\rho^2}{2}}\rho^2d\zeta'dyd\rho\\
&-\int_0^\infty\int_{\mathbb{R}^3}\int_{\Omega}\mbox{div}_y\left( \big[ e^{-\frac{\nu(\rho)}{\rho}|x-y|}k \left( \rho\frac{x-y}{|x-y|},\zeta' \right)\frac{(x-y)\cdot n(x)}{|x-y|^3}e^{-\frac{\rho^2}{2}}\rho^2f(x,\zeta) \big] v\right)
 \\
 &\quad dyd\zeta d\rho\\&+\int_0^\infty\int_{\mathbb{R}^3}\int_{\Omega}e^{-\frac{\nu(\rho)}{\rho}|x-y|} k \left(\rho\frac{x-y}{|x-y|},\zeta' \right)\frac{(x-y)\cdot\nabla^x_v(n(x))}{|x-y|^3}e^{-\frac{\rho^2}2}\rho^2f(x,\zeta')dyd\zeta'd\rho
 \\
 &
=:\nabla^x D_f^1+\nabla^x D_f^2+\nabla^x D_f^3.
\end{split}\end{equation}
Notice that since $\partial\Omega $ is smooth, we can see that $\nabla^xD_f^3$ is bounded.
Regarding $\nabla^x D_f^1$, by direct calculation, we have
\begin{equation}\begin{split}
\\&  \nabla^x_vD_f^1=\Big|\int_0^\infty\int_{\mathbb{R}^3}\int_{\Omega}e^{-\frac{\nu(\rho)}{\rho}|x-y|}k \left(\rho\frac{x-y}{|x-y|},\zeta' \right)e^{-\frac{\rho^2}{2}}\rho^2[f(y,\zeta')-f(x,\zeta')]\\&\quad \cdot\Big[-\frac{\nu(\rho)}{\rho}\frac{(x-y)\cdot v}{|x-y|}\frac{(x-y)\cdot n(x)}{|x-y|^3}+\frac{v\cdot n(x)}{|x-y|^3}+\frac{(x-y)\cdot\nabla^x_vn(x)}{|x-y|^3}\\&\quad-3\frac{(x-y)\cdot n(x)}{|x-y|^4}\frac{(x-y)\cdot v}{|x-y|}\Big]
\\&+[f(y,\zeta')-f(x,\zeta')]e^{-\frac{\nu(\rho)}{\rho}|x-y|}\frac{(x-y)\cdot n(x)}{|x-y|^3}\\&\quad\cdot e^{-\frac{\rho^2}2}\rho^3\left(\mbox{grad}_\zeta k \left( \rho\frac{x-y}{|x-y|},\zeta' \right)\right) \cdot \left( \frac{v}{|x-y|}-\frac{(x-y)\cdot v}{|x-y|^3}(x-y) \right)dyd\zeta'd\rho\Big|\\ &\leq C \int_0^\infty\int_{\mathbb{R}^3}\int_{\Omega} \left\{\frac1{|\rho\frac{x-y}{|x-y|}-\zeta'|}e^{-\frac18|\rho\frac{x-y}{|x-y|}-\zeta'|^2} \left[\frac{1}{|x-y|^{\frac32+\frac{\epsilon}2}}+\frac1{|x-y|^{\frac52+\frac{\epsilon}2}} \right] \rho^2 \right. \\&+ \left. \frac{1+|\rho\frac{x-y}{|x-y|}|}{|\rho\frac{x-y}{|x-y|}-\zeta'|^2}e^{-\frac18|\rho\frac{x-y}{|x-y|}-\zeta'|^2} \frac{1}{|x-y|^{\frac52+\frac{\epsilon}2}}\rho^3 \right\} (1+|\zeta'|^{-1})e^{-\frac{\rho^2}{2}}dyd\zeta'd\rho\\ &\leq C\int_{\mathbb{R}^3}\int_{\mathbb{R}^3}\int_0^R\frac{1}{|\zeta-\zeta'|}e^{-\frac18|\zeta-\zeta'|^2} \left(1+\frac1{|\zeta'|} \right) e^{-\frac{|\zeta|^2}2}[r^{\frac12-{\frac{\epsilon}2}}+r^{-\frac12-\frac{\epsilon}2}]\\&\quad+\left(\frac{(1+|\zeta|)^2}{|\zeta-\zeta'|^2}+\frac{1+|\zeta|}{|\zeta-\zeta'||\zeta'|}\right)e^{-\frac18|\zeta-\zeta'|^2}e^{-\frac{|\zeta|^2}2}r^{-\frac12-\frac{\epsilon}2}drd\zeta'd\zeta\\&\leq C.
\end{split}\end{equation}
In the above derivation,  in addition to \eqref{estimateK} and \eqref{gradianK}, we have used the fact\begin{equation}e^{-\frac{\nu(\rho)}{\rho}|x-y|}\frac{\nu(\rho)}{\rho}\leq C|x-y|^{-1}\end{equation} and the triangle inequality $|\zeta|\leq |\zeta'|+|\zeta'-\zeta|$.
Next, we shall prove that $\nabla^x D_f^2$ is bounded in the senses of improper integral. We define $\Omega_\epsilon=\Omega\setminus B(x,\epsilon)$ and name the corresponding integral  $\nabla^x D_f^{2,\epsilon}$. Applying the divergence theory, we have
\begin{equation}
\begin{split}
\nabla^x D_f^{2,\epsilon}=&-\int_0^\infty\int_{\mathbb{R}^3}\int_{\partial\Omega\setminus B(x,\epsilon)}e^{-\frac{\nu(\rho)}{\rho}|x-y|} k \left(\rho\frac{x-y}{|x-y|},\zeta' \right)\frac{(x-y)\cdot n(x)}{|x-y|^3}\\
&\quad e^{-\frac{\rho^2}{2}}\rho^2f(x,\zeta')[v\cdot n(y) ]dA(y)d\zeta' d\rho\\
&-\int_0^\infty\int_{\mathbb{R}^3}\int_{\partial B(x,\epsilon) \cap \Omega}e^{-\frac{\nu(\rho)}{\rho}|x-y|} k \left(\rho\frac{x-y}{|x-y|},\zeta' \right)\frac{(x-y)\cdot n(x)}{|x-y|^3}\\
&\quad e^{-\frac{\rho^2}{2}}\rho^2f(x,\zeta') \left[ v\cdot \frac{x-y}{|x-y|} \right]dA(y)d\zeta' d\rho\\
&=:S^\epsilon+B^\epsilon.
\end{split}
\end{equation}

For $S^\epsilon$, we further break the domain of integration by $GB(x,r_1),$ where $r_1$ is as defined in Proposition~\ref{propositionA}. It is not hard to see that the integral outside $GB(x,r_1)$ is bounded.  Inside the $GB(x,r_1)$, by applying Lemma \ref{lemmaB}, we have 
\begin{equation}\begin{split}
&\Big|\int_0^\infty\int_{\mathbb{R}^3}\int_{GB(x,r_1)\setminus B(x,\epsilon)}e^{-\frac{\nu(\rho)}{\rho}|x-y|}k\left(\rho\frac{x-y}{|x-y|},\zeta' \right)\frac{(x-y)\cdot n(x)}{|x-y|^3}\\
&\quad e^{-\frac{\rho^2}{2}}\rho^2f(x,\zeta')[v\cdot n(y) ]dA(y)d\zeta' d\rho \Big|\\
&\leq C\Vert f\Vert_{L^\infty_{x,\zeta}} \int_0^\infty\int_{GB(x,r_1)}e^{-\frac{\rho^2}{2}}\rho^2dyd\rho\leq  C\Vert f\Vert_{L^\infty_{x,\zeta}} .
\end{split}\end{equation}

We are going to   deal with $B^{\epsilon}$ and will see that it  in fact forms a "residue".
We introduce  spherical coordinates on $B(x,\epsilon)$ so that $-n(x)$ is the north pole so that
 \begin{equation}\hat{\zeta}:=\frac{x-y}{|x-y|}=\sin\theta\cos\phi v+\sin\theta\sin\phi( n(x)\times v)+\cos\theta n(x).\end{equation}
We  use $D_\epsilon$ to denote the domain in the chart that maps to $\partial B(x,\epsilon)\cap  \Omega$.
Let 
 \begin{equation}D'_\epsilon:=\{\rho[\sin\theta\cos\phi v+\sin\theta\sin\phi( n(x)\times v)+\cos\theta n(x))]|(\theta,\phi)\in D_\epsilon, \rho>0\}.\end{equation}

We have

\begin{equation}
\begin{split}
 B^\epsilon=&-\int_0^\infty\int_{\mathbb{R}^3}\int_0^\pi\int_0^{2\pi}
\chi_{D_\epsilon}(\theta,\phi)e^{-\frac{\nu(\rho)}{\rho}\epsilon}k(\rho\hat{\zeta},\zeta')e^{-\frac{\rho^2}{2}}\rho^2f(x,\zeta')\cos\theta\sin^2\theta\cos\phi d\phi d\theta d\zeta'd\rho\\=&-\int_{\mathbb{R}^3}\int_{\mathbb{R}^3}
\chi_{D'_\epsilon}(\zeta)e^{-\frac{\nu(|\zeta|)}{|\zeta|}\epsilon}k(\zeta,\zeta')e^{-\frac{|\zeta|^2}{2}}f(x,\zeta')\frac{\zeta\cdot n(x)}{|\zeta|}\frac{\zeta\cdot v }{|\zeta|} d\zeta d\zeta'
\end{split}
\end{equation}
We can conclude from the dominated convergence theorem that
\begin{equation}
\lim_{\epsilon\to0}B^\epsilon=-\int_{\zeta\cdot n(x)\leq0}\int_{\mathbb{R}^3}k(\zeta,\zeta')e^{-\frac{|\zeta|^2}{2}}f(x,\zeta')\frac{\zeta\cdot n(x)}{|\zeta|}\frac{\zeta\cdot v }{|\zeta|} d\zeta d\zeta'
\end{equation}

We notice that \begin{equation}
\begin{split}
|\nabla^x D_f^{2,\epsilon}|&\leq C\Vert f\Vert_{L^\infty_{x,\zeta}}\int_0^\infty\int_{\Omega}\frac{1}{|x-y|^2}e^{-\frac{\rho^2}2}\rho^2dydd\rho
 \\&\leq  C\Vert f\Vert_{L^\infty_{x,\zeta}}.
\end{split}\end{equation}
We conclude the lemma.

\section{Differentiability of $f$ \label{SDiffVDF}}

The main result of this article is  summarized in the following lemma.
\begin{lem} Under the assumption of Theorem \ref{MainThm}, regarding the solution of \eqref{SBE}, we have the following estimates
\begin{align}
\left| \frac{\partial}{\partial x_i}f(x,\zeta) \right|&\leq C(1+d_x^{-1})^{\frac43+\epsilon}, \label{diffxv}\\
\label{diffzeta} \left| \frac{\partial}{\partial \zeta_i}f(x,\zeta) \right|&\leq C(1+d_x^{-1})^{\frac43+\epsilon}.
\end{align} 
\end{lem}
This section is devoted to the proof of \eqref{diffxv}. 
We leave the proof of \eqref{diffzeta} to the next section. 
In view of \eqref{def123},  to prove \eqref{diffxv}, we shall proceed the estimates of $\frac{\partial}{\partial x_i} I(x,\zeta)$, $\frac{\partial}{\partial x_i} II(x, \zeta)$ and $\frac{\partial}{\partial x_i} III (x, \zeta)$ respectively.
We  first show that $I$ and $II$ defined in \eqref{def123}  preserve the regularity from boundary.  
 \begin{lem}\label{DiffIandII} Let $\Omega$ be the domain introduced before. 
Suppose  there exist constants $a,\  M>0$ such that
  \begin{align}
\label{fboundary1}  &|\nabla^x_{\eta}f(X,\zeta)|\leq M|\eta|e^{-a|\zeta|^2},\\
\label{fboundary2}  &|f(X,\zeta)|\leq Me^{-a|\zeta|^2},
  \end{align} 
  for all $(X,\zeta)\in\Gamma_-$ and $\eta \in T_X (\partial\Omega)$.
  Then, for $x\in\Omega$, the following estimates hold 
  \begin{align}
  &\left| \frac{\partial}{\partial x_i}I(x,\zeta) \right|\leq C d_x^{-1} e^{-\frac{a}{2}|\zeta|^2},\label{1infest}\\
  &\left| \frac{\partial}{\partial x_i}II(x,\zeta) \right|\leq  C(1+d_x^{-1})^{\frac43+\epsilon'}, \label{IIdiffest} \\
  &\left| \frac{\partial}{\partial x_i} II(x,\zeta) \right|\leq  C(1+d_x^{-1})^{\frac13+\epsilon'} \left( 1+\frac1{|\zeta|} \right).\label{II1overzeta}\end{align} 
\end{lem}

To prove Lemma~\ref{DiffIandII}, we have the following observation.
\begin{prop} \label{Dtau}
Let $\tau_-(x,\zeta)$ and $p(x, \zeta)$ be as defined in \eqref{deftauminus} and \eqref{defpp}. Then
\begin{align}
\label{ddd1} |\frac{\partial}{\partial x_i}p(x,\zeta)|&\leq \frac{1}{N(x,\zeta)},\\
\label{ddd2}|\frac{\partial}{\partial x_i}\tau_-(x,\zeta)|&\leq \frac{2}{N(x,\zeta)|\zeta|},\\
\label{ddd3}\tau_-(x,\zeta)&\geq \frac{d_x}{N(x,\zeta)|\zeta|}.
\end{align}
\end{prop}
We  notice that  
\begin{enumerate}
\item[(a)] \eqref{ddd1} is a direct result of \eqref{prop11} of  Proposition \ref{Propdifference}.
\item[(b)]  \eqref{ddd2} is derived from \eqref{defpp}, \eqref{ddd1} and  the fact that $N(x,\zeta) \le 1$.
\item[(c)] \eqref{ddd3} is a direct result of \eqref{defpp} and \eqref{prop13} of  Proposition \ref{Propdifference}.
\end{enumerate}
 We are now in a position   to prove \eqref{1infest}.

\begin{proof}[proof of \eqref{1infest}]

Let $e_i$ be the $i$-th unit vector in $\mathbb R^3$. Formal calculation gives
\begin{equation}\begin{split}
 \frac{\partial}{\partial x_i}I(x,\zeta)&=\left.\frac{d}{dt}\left( f(p(x+te_i,\zeta),\zeta)e^{-\nu(|\zeta|)\tau_-(x+te_i,\zeta)}\right)\right|_{t=0}\\ &=\left.\frac{d}{dt}\left( f(p(x+te_i,\zeta),\zeta)\right)\right|_{t=0}e^{-\nu(|\zeta|)\tau_-(x,\zeta)}\\&\quad\left.-  \frac{\partial}{\partial x_i} \tau_-(x,\zeta) \right.\nu(|\zeta|)f(p(x,\zeta),\zeta)e^{-\nu(|\zeta|)\tau_-(x,\zeta)}\end{split}
\end{equation}
Note that by \eqref{fboundary1} and \eqref{ddd1}, we have
\begin{equation}
\left| \left.\frac{d}{dt}\left( f(p(x+te_i,\zeta),\zeta)\right)\right|_{t=0} \right| \le \frac{M}{N(x,\zeta)} e^{-a|\zeta|^2}.
\end{equation}
Therefore, summing up the above two equations and applying Proposition~\ref{Dtau}, we obtain
\begin{align}
\nonumber \left| \frac{\partial}{\partial x_i}I(x,\zeta) \right|&\leq C \Big(\frac{1}{N(x,\zeta)}e^{-\nu(|\zeta|)\frac{d_x}{N(x,\zeta)|\zeta|}}\\
\label{1113}&\,\,+\nu(|\zeta|)\frac{1}{N(x,\zeta)|\zeta|}e^{-\nu(|\zeta|)\frac{d_x}{N(x,\zeta)|\zeta|}}\Big)e^{-a|\zeta|^2}\\
\label{1114}&\leq C d_x^{-1}(|\zeta|+1)e^{-a|\zeta|^2}\leq C d_x^{-1} e^{-\frac{a}{2}|\zeta|^2}.
\end{align}
\end{proof}

Taking the derivative on $II$ with respect to $x_i$, we have
\begin{equation}\begin{split}
\frac{\partial}{\partial x_i} II(x,\zeta)&=\int_0^{\tau_-(x,\zeta)}e^{-\nu(|\zeta|)s}\int_{\mathbb{R}^3} k(\zeta,\zeta')\frac{\partial}{\partial x_i}I(x-s\zeta ,\zeta')d\zeta'ds\\&\quad+e^{-\nu(|\zeta|)\tau_-(x,\zeta)}\int_{\mathbb{R}^3}k(\zeta,\zeta')f(p(x,\zeta),\zeta')d\zeta' \frac{\partial}{\partial x_i} \tau_-(x,\zeta)\\&=:  II^A+II^B.\end{split}
\end{equation}
By Proposition \ref{Dtau}, we see that
\begin{equation}
\label{IIB}
|II^B|\leq C\Vert f\Vert_{L^\infty_{x,\zeta}}e^{-\nu_0\frac{d_x}{N(x,\zeta)|\zeta|}}\frac1{N(x,\zeta)|\zeta|}\leq C\Vert f\Vert_{L^\infty_{x,\zeta}}d_x^{-1}.
\end{equation}
We let
\begin{equation}
\label{defofH}
H(x,\zeta)=\int_{\mathbb{R}^3} k(\zeta,\zeta') I(x,\zeta')d\zeta'.
\end{equation}
Then, we have 
\begin{equation}
II^A=\int_0^{\tau_-(x,\zeta)}e^{-\nu(|\zeta|)s}\frac{\partial}{\partial x_i} H(x-s\zeta ,\zeta')ds
\end{equation}
Concerning $\frac{\partial}{\partial x_i}H$, we have the following estimate.
\begin{prop}
\label{Hprop}
Let $H$ be as defined in \eqref{defofH}. We have
\begin{equation}\label{DvH13}
\left| \frac{\partial}{\partial x_i}H(x,\zeta) \right|\leq C d_x^{-\frac13} \left( \big|\ln d_x \big|+1 \right).
\end{equation}
\end{prop}
\begin{proof}

In order to get a good estimate near the boundary, we shall break the domain of integration into two, $B(\zeta, d_x^{\frac13})$ and $\mathbb{R}^3\setminus B(\zeta,d_x^{\frac13})$, and name the corresponding integrals $DH_s$ and $DH_l$ respectively.
 For the estimate of $DH_s$, by applying the estimate \eqref{1infest}, we obtain
 \begin{equation}
 \label{dhs}
 \begin{split}
 |DH_s|&\leq \left| \int_{ B(\zeta,d_x^{\frac13})}  k(\zeta,\zeta')\frac{\partial}{\partial x_i}I(x - s \zeta,\zeta') d\zeta'  \right| \\
 &\leq C\frac{1}{d_x} \int_{ B(\zeta,d_x^{\frac13})} | k(\zeta,\zeta')|d\zeta'\\&\leq C\frac{1}{d_x} \int_{ B(\zeta,d_x^{\frac13})} \frac1{|\zeta-\zeta'|}d\zeta'\\&\leq C\frac{1}{d_x} \int_0^{\pi}\int_0^{d_x^{\frac13}}\frac1\rho\rho^2\sin\theta d\rho d\theta\\&\leq C{d_x^{-\frac13}}. \end{split}
 \end{equation}
 Regarding the estimate of $DH_l$,  we first notice that \eqref{1113} and \eqref{estimatenu}  imply 
 \begin{equation}\begin{split}
&\left| \frac{\partial }{ \partial x_i} I(x,\zeta') \right|\leq C \left(\frac{1}{N(x,\zeta')}+\frac{1}{N(x,\zeta')|\zeta|}\right)e^{-a|\zeta'|^2},\end{split}
\end{equation}
since $0 \le \gamma < 1.$
Next, we consider  the coordinate change $\zeta ' = l (y-x)$ as we employed in Section 1 together with the fact $|k(\zeta,\zeta')|<Cd_x^{-\frac13}$ in the domain  $\mathbb{R}^3\setminus B(\zeta,d_x^{\frac13})$ to obtain
 \begin{equation}
 \label{dhl}
 \begin{split}
 &|DH_l|\leq C d_x^{-\frac13}\int_0^\infty\int_{\partial\Omega} \left|\frac{\partial}{\partial x_i} I(x,\zeta') \right|l^2|n(y)\cdot(x-y)|dA(y)dl\\
 &\leq C d_x^{-\frac13}\int_0^\infty\int_{\partial\Omega}\left(\frac{1}{N(x,\zeta')}+\frac{1}{N(x,\zeta')l|x-y|}\right)\\
 &\quad\quad \quad \times e^{-al^2|x-y|^2}l^2|n(y)\cdot(x-y)|dA(y)dl\\
 &\leq C d_x^{-\frac13}\int_0^\infty\int_{\partial\Omega}\left(\frac{|x-y|}{|n(y)\cdot(x-y)|}+\frac{1}{l|n(y)\cdot(x-y)|}\right)\\
 &\quad\quad \quad \times e^{-al^2|x-y|^2}l^2|n(y)\cdot(x-y)|dA(y)dl\\ 
&\leq C d_x^{-\frac13}\int_0^\infty\int_{\partial\Omega}e^{-al^2|x-y|^2}(l^2|x-y|+l)dA(y)dl\\
&\leq Cd_x^{-\frac13}\int_{\partial\Omega}\int_0^\infty e^{-az^2}(z^2+z)dz\frac1{|x-y|^2}dA(y)dz\\
&\leq Cd_x^{-\frac13}\int_{\partial\Omega}\frac1{|x-y|^2}dA(y)\\
&\leq Cd_x^{-\frac13}(|\ln d_x|+1).
 \end{split}
 \end{equation}
Notice that  we let $z=|x-y|l$ in the third last line in the above derivation and applied  Lemma \ref{lemmaA} to draw our conclusion in the last line. 
Combining \eqref{dhs} and \eqref{dhl}, the proof of Proposition \ref{Hprop} is complete.
\end{proof}
 \begin{proof}[proof of \eqref{IIdiffest}  and \eqref{II1overzeta}. ]
 
 Now, we set  $l=|x-p(x,\zeta)|$. Applying Proposition~\ref{pointtobd} and Proposition~\ref{Hprop}, we obtain 
\begin{equation}\begin{split}
|II^A|&\leq C\int_0^{\tau_-(x,\zeta)}e^{-\nu(|\zeta|)s} d_{(x-s\zeta)}^{-(\frac13+\epsilon')} ds\\
&\leq C \left( \left( \frac{d_{x }}2 \right)^{-(\frac13+\epsilon')} \int_0^{\frac{d_x}{2|\zeta|}}e^{-\nu(|\zeta|)s}ds \right.\\
&\qquad + \left. \int_{\frac{d_x}2}^{l}e^{-\frac{\nu(|\zeta|) r}{|\zeta|}}d_x^{-(\frac13+\epsilon')}|l-r|^{-(\frac13+\epsilon')}\frac{1}{|\zeta|}dr \right).\end{split}\end{equation}
We observe that
\begin{equation*}
\begin{aligned}
\int_{\frac{d_x}2}^{l} & e^{-\frac{\nu(|\zeta|) r}{|\zeta|}}d_x^{-(\frac13+\epsilon')}|l-r|^{-(\frac13+\epsilon')}\frac{1}{|\zeta|}dr \\
&\le \left\{ \begin{aligned} & C \int_{\frac{d_x}2}^{l}d_x^{-(\frac13+\epsilon')}|l-r|^{-(\frac13+\epsilon')}\frac{1}{r}dr\\
& \int_{\frac{d_x}2}^{l}d_x^{-(\frac13+\epsilon')}|l-r|^{-(\frac13+\epsilon')}\frac{1}{|\zeta|}dr.\end{aligned}\right.\end{aligned} 
\end{equation*}
Hence, 
\begin{equation}
\label{IIA}
|II^A| \leq \left\{ \begin{aligned} & C(1+ d_x^{-1})^{\frac{4}{3}+\epsilon'}\\
& C(1+ d_x^{-1})^{\frac{1}{3}+\epsilon'}\frac{1}{|\zeta|}.
\end{aligned}\right.
\end{equation}
Finally, combining \eqref{IIB} and \eqref{IIA}, we obtain \eqref{IIdiffest}  and \eqref{II1overzeta}. This completes the proof of Lemma~\ref{DiffIandII}.
\end{proof}

Combining Lemma \ref{DiffIandII} and Proposition \ref{HolderIII}, we have a refined estimate\begin{equation}\label{fpointHolder}\left|f(x,\zeta)-f(y,\zeta)\right|\leq C\Vert  f\Vert_{L^\infty_{x,\zeta}}(1+d_x^{-1})^{\frac4{3}+\epsilon}|x-y|^{1-\epsilon}\end{equation}
in case $|x-y|<\frac{d_x}2$.

 We are now in a position to perform bootstrapping the regularity.

\begin{lem} Let $f\in L^\infty_{x,\zeta}$ be a  stationary solution to the linearized Boltzmann equation and $x$ be an interior point of $\Omega$. Suppose that there exist $0<\sigma<1$, \ $0<\delta<d(x,\partial\Omega)$,  and $ M>0$  such that, 
\begin{equation}
\label{holderf}
\sup_{\zeta\in\mathbb{R}^3}|f(x,\zeta)-f(y,\zeta)|\leq M|x-y|^\sigma,
\end{equation}
whenever $y\in B(x,\delta)$.

Then, $G$ is differentiable at $x$. Furthermore,\begin{equation}\label{diffGestimate}
\left| \frac{\partial}{\partial x_i}G(x,\zeta) \right| \leq C(\Vert f\Vert_{L^\infty_{x,\zeta}}(1+ |\ln \delta|) +M\delta^{\sigma}).
\end{equation}
\end{lem}

\begin{proof}

Recall that
\begin{equation}\label{Gformula22}
\begin{split}
G(x,\zeta)=\int_0^\infty\int_{\Omega}\int_{\mathbb{R}^3}& k \left( \zeta,\rho\frac{(x-y)}{|x-y|} \right) e^{-\nu(\rho)\frac{|x-y|}{\rho}} \\
& \times k \left( \rho\frac{(x-y)}{|x-y|},\eta \right)f(y,\eta)\frac{\rho}{|x-y|^2}d\eta dyd\rho.
\end{split}
\end{equation}
 We first formally differentiate the above formula with respect to  $x_i$ and divide the domain of integration into two parts : $ B(x,\delta)$ and $\Omega \setminus B(x,\delta) $. We denote the corresponding integrals as $g_s$ and $g_l$. 
 Regarding the estimate of $g_l$, the typical term is
 \begin{equation}
\begin{split}
|g_{l1}| := &\left| -2 \int_0^\infty\int_{\Omega \setminus B(x,\delta) }\int_{\mathbb{R}^3} k \left( \zeta,\rho\frac{(x-y)}{|x-y|} \right)e^{-\nu(\rho)\frac{|x-y|}{\rho}} \right.\\
& \left. \times k \left( \rho\frac{(x-y)}{|x-y|},\eta \right)f(y,\eta)\frac{\rho}{|x-y|^3} \frac{x_i-y_i}{|x-y|}d\eta dyd\rho \right|\\
&\le C\Vert f\Vert_{L^\infty_{x,\zeta}} \int_0^{\infty} \int_{\Omega \setminus B(x,\delta)}  k \left(\zeta, \rho\frac{(x-y)}{|x-y|} \right) \frac{\rho}{|x-y|^3} dy d \rho \\
&\le C\Vert f\Vert_{L^\infty_{x,\zeta}} \int_0^{\infty} \int_{S^2}\int_{\delta}^R  k(\zeta,\rho\omega) \frac{\rho}{r^3}r^2 dr d \omega d \rho d\omega\\
& \le  C\Vert f\Vert_{L^\infty_{x,\zeta}}  (1 + |\ln \delta|),
\end{split}
\end{equation}  
where $R$ is the diameter of $\Omega$. By using \eqref{CH}, in the same fashion, we readily see that
\begin{equation}
\label{gl}
| g_l | \le C\Vert f\Vert_{L^\infty_{x,\zeta}}  (1 + |\ln \delta|).
\end{equation}

 Regarding the estimate of $g_{s}$, in order to utilize the H\"{o}lder continuity, we  subtract and add  $f(x,\eta)$ in the integrand as follows:
 
\begin{equation}\label{gformula1}
\begin{split}
&g_{s}(x,\zeta)=\int_0^\infty\int_{B(x,\delta)}\int_{\mathbb{R}^3 }\\&\quad \left(f(y,\eta)-f(x,\eta)\right) \frac{\partial}{\partial x_i}\left[ k \left(\zeta,\rho\frac{x-y}{|x-y|} \right)e^{-\nu(\rho)\frac{|x-y|}{\rho}} k \left(\rho\frac{x-y}{|x-y|},\eta \right)\frac{\rho}{|x-y|^2}\right]d\eta dyd\rho\\&+ \int_0^\infty\int_{B(x,\delta)}\int_{\mathbb{R}^3} f(x,\eta) \frac{\partial}{\partial x_i}\left[ k \left(\zeta,\rho\frac{x-y}{|x-y|} \right)e^{-\nu(\rho)\frac{|x-y|}{\rho}} k \left(\rho\frac{x-y}{|x-y|},\eta \right)\frac{\rho}{|x-y|^2}\right]d\eta dyd\rho\\ &=: g_{s1} +g_{s2}.
\end{split}
\end{equation}
For $g_{s1}$, the H\"{o}lder continuity of $f$ in space variables, see \eqref{holderf}, makes the integrand  integrable. We have
\begin{equation}
\label{gs1}
|g_{s1}|\leq C M\delta^\sigma.
\end{equation} 

For $g_{s2}$, we first remove an $\epsilon$-ball and integrate:\begin{equation}
\begin{split}
&g_{s2}^{\epsilon}:=\int_0^\infty\int_{B(x,\delta)\setminus B(x,\epsilon)}\int_{\mathbb{R}^3 } f(x,\eta)\\
&\quad \frac{\partial}{\partial x_i}\left[ k \left(\zeta,\rho\frac{x-y}{|x-y|} \right) e^{-\nu(\rho)\frac{|x-y|}{\rho}} k \left(\rho\frac{x-y}{|x-y|},\eta \right)\frac{\rho}{|x-y|^2}\right]d\eta dydd\rho\\
 &=\int_0^\infty\int_{B(x,\delta)\setminus B(x,\epsilon)}\int_{\mathbb{R}^3 } f(x,\eta)\\&\quad \left(-\left(\frac{\partial}{\partial y_i}\left[ k\left(\zeta,\rho\frac{x-y}{|x-y|} \right)e^{-\nu(\rho)\frac{|x-y|}{\rho}} k \left(\rho\frac{x-y}{|x-y|},\eta \right)\frac{\rho}{|x-y|^2}\right]\right)d\eta dydd\rho \right.\\
 &=\int_0^\infty\int_{\mathbb{R}^3 }\left(-\int_{S^2(x,\delta)}+\int_{S^2(x,\epsilon)}\right) f(x,\eta)\\
 &\quad \left[ k \left(\zeta,\rho\frac{x-y}{|x-y|} \right)e^{-\nu(\rho)\frac{|x-y|}{\rho}} k \left(\rho\frac{x-y}{|x-y|},\eta \right)\frac{\rho}{|x-y|^2}\right] n_i(y)dA(y)d\eta dd\rho\\&=\int_{\mathbb{R}^3}\int_{\mathbb{R}^3} f(x,\eta)k(\zeta,\zeta')\left(- e^{-\frac{\nu(|\zeta'|)}{|\zeta'|}\delta}+ e^{-\frac{\nu(|\zeta'|)}{|\zeta'|}\epsilon}\right)k(\zeta',\eta)\frac{1}{|\zeta'|}\frac{\zeta'_i}{|\zeta'|}d\zeta'd\eta.\end{split}
\end{equation}

Notice that the integrand in the last integral above is bounded by
\begin{equation}2\Vert f\Vert_{L^\infty_{x,\zeta}}
k(\zeta,\zeta')k(\zeta',\eta)|\zeta'|^{-1},
\end{equation}
which is integrable in $(\zeta',\eta)$. Therefore, we can pass the limit $\epsilon\to0$ and, furthermore

\begin{equation}
\label{gs2}
|g_{s2}| \leq C \Vert f\Vert_{L^\infty_{x,\zeta}}.
\end{equation}
Inferring from \eqref{gl}, \eqref{gs1} and \eqref{gs2}, we obtain \eqref{diffGestimate}. 
\end{proof} 

 Finally, to complete the proof of \eqref{diffxv}, we now estimate $\frac{\partial}{\partial x_i} III (x, \zeta)$.
{\sc Proof of \eqref{diffxv}.}
\begin{proof}
Differentiating $III$ directly and applying  \eqref{ddd2}, \eqref{ddd3}, \eqref{fpointHolder} and \eqref{diffGestimate}, we obtain

\begin{equation}\begin{split}
\left| {\frac{\partial}{\partial x_i}}III(x,\zeta) \right|&\le \left|\int_0^{\tau_-(x,\zeta)}e^{-\nu(|\zeta|)s}{\frac{\partial}{\partial x_i}} G(x-s \zeta,\zeta)ds\right|\\
& + \Big|  \frac{\partial}{\partial x_i} \tau_-(x,\zeta) e^{-\nu(|\zeta|)\tau_-(x, \zeta)} G( p(x,\zeta) ,\zeta)   \Big|\\
&\leq C\Vert f\Vert_{L^\infty_{x,\zeta}}\left|\int_0^{\tau_-(x,\zeta)}e^{-\nu(|\zeta|)s}(1+d_{x-s\zeta}^{-\frac13-2\epsilon})ds \right|\\
& + C\Vert f\Vert_{L^\infty_{x,\zeta}} d_x^{-1}.
\end{split}\end{equation}
Letting $l=|x-p(x,\zeta)|$ and  applying Proposition~\ref{pointtobd}, we readily obtain
\begin{equation}\begin{split}
&\left|\int_0^{\tau_-(x,\zeta)}e^{-\nu(|\zeta|)s}(1+d_{x-s\zeta}^{-\frac13-2\epsilon})ds \right|\\
&\leq C \left( 1+d_x^{-\frac13-2\epsilon} \right)\int_0^{\frac{d_x}{2|\zeta|}}e^{-\nu(|\zeta|)s}ds+\int_{\frac{d_x}2}^{l}e^{-\frac{\nu(|\zeta|)}{|\zeta|}r} \left( 1+d_{x-r\hat{\zeta}}^{-\frac13-2\epsilon} \right)\frac1{|\zeta|}dr \\
&\leq C \left(1+d_x^{-\frac13-2\epsilon} \right)+ Cd_x^{-1}\int_{\frac{d_x}2}^{l} \left( 1+d_{x}^{-\frac13-2\epsilon}|l-r|^{-\frac13-2\epsilon} \right)\,dr\\&\leq C \left(1+d_x \right)^{-\frac43-2\epsilon}.\end{split}\end{equation}
Therefore, we see that
\begin{equation}
\label{iiix}
\left| {\frac{\partial}{\partial x_i}}III(x,\zeta) \right| \le C\Vert f\Vert_{L^\infty_{x,\zeta}} (1+d_x)^{-\frac43-2\epsilon}.
\end{equation}
Taking \eqref{1infest}, \eqref{IIdiffest} and \eqref{iiix} into account, by \eqref{def123}, the proof of \eqref{diffxv} is complete.\end{proof}

 Notice that 
 \begin{equation*}
  \int_{\frac{d_x}2}^{l}e^{-\frac{\nu(|\zeta|)}{|\zeta|}r} \left(1+d_{x-r\hat{\zeta}}^{-\frac13-2\epsilon} \right)\frac1{|\zeta|}dr \le
\left\{  \begin{aligned}
  & C(1+d_x)^{-\frac13-2\epsilon}\frac{1}{|\zeta|} \\
 & Cd_x^{-1}\int_{\frac{d_x}2}^{l} \left(1+d_{x}^{-\frac13-2\epsilon}|l-r|^{-\frac13-2\epsilon} \right)\,dr.
  \end{aligned}
 \right.
 \end{equation*}
 
 Hence, if we allow the singularity at $\zeta=0$, we can estimate
   \begin{equation}\begin{split}\label{dviiising}
\left|{\frac{\partial}{\partial x_i}}III(x,\zeta) \right|&\leq C(1+d_x)^{-\frac13-2\epsilon} \left(1+\frac{1}{|\zeta|} \right).\end{split}\end{equation} 


\section{Derivative with respect to  velocity \label{zetaDiff}}

In this section, we shall discuss the derivative of the velocity distribution function with respect to microscopic velocity $\zeta$, i.e., \eqref{diffzeta}. Differentiating the integral equation (\ref{inteq}) directly, we have
\begin{equation}\begin{split}&\frac{\partial}{\partial \zeta_i} f(x,\zeta)=e^{-\nu(|\zeta|)\tau_-(x,\zeta)} \left[\nabla^x_{\frac{\partial}{\partial \zeta_i}p(x,\zeta)}f(p(x,\zeta),\zeta)+\frac{\partial}{\partial \zeta_i} f(p(x,\zeta),\zeta)\right]\\&+f(p(x,\zeta),\zeta)e^{-\nu(|\zeta|)\tau_-(x,\zeta)}
\left[-\nu'(|\zeta|)\frac{\zeta_i}{|\zeta|} \tau_-(x,\zeta)-\nu(|\zeta|)\frac{\partial \tau_-(x,\zeta)}{\partial \zeta_i}\right]\\&+ \frac{\partial\tau_-(x,\zeta)}{\partial \zeta_i}e^{-\nu(|\zeta|)\tau_-(x,\zeta)}K(f)(p(x,\zeta),\zeta)\\&+\int_0^{\tau_-(x,\zeta)}-\nu'(|\zeta|)\frac{\zeta_i}{|\zeta|}e^{-\nu(|\zeta|)s}K(f)(x-s\zeta,\zeta)ds\\&-\int_0^{\tau_-(x,\zeta)}e^{-\nu(|\zeta|)s}\frac{\partial}{\partial x_i}K(f)(x-s\zeta,\zeta)sds\\&+\int_0^{\tau_-(x,\zeta)}e^{-\nu(|\zeta|)s}\frac{\partial}{\partial \zeta_i}K(f)(x-s\zeta,\zeta)ds\\&=:D_v^1+D_v^2+D_v^3+D_v^4+D_v^5+D_v^6.
\end{split}\end{equation}
From the fact $\nu'(|\zeta|)$ is bounded, we can conclude that $D_v^4$ is bounded. 
Using the estimate \begin{equation}
\left\Vert \frac{\partial}{\partial\zeta_i}K(f) \right\Vert_{L^\infty_\zeta}\leq C\Vert f\Vert_{L^\infty_\zeta}, \end{equation}
we can prove that $D_v^6$ is bounded.
Using \eqref{DvH13}, \eqref{II1overzeta}, and \eqref{dviiising}, we  have
\begin{equation}
\left| \frac{\partial}{\partial x_i}K(f)(x-s\zeta,\zeta) \right|\leq C \left(1+d_{x-s\zeta}^{-1} \right)^{\frac13+\epsilon}.
\end{equation} 
Then, applying Proposition\ref{pointtobd}, we have
\begin{equation}\begin{split}
|D_v^5|&\leq C\int_0^{\tau_-(x,\zeta)}e^{-\nu(|\zeta|)s}\big(1+[d_x|\zeta|(\tau_-(x,\zeta)-s)]^{-\frac13-\epsilon}\big)sds\\&\leq C\int_0^{l}e^{-\frac{\nu(|\zeta|)}{|\zeta|}r}\big(1+[d_x|(l-r))]^{-\frac13-\epsilon}\big)\frac{r}{|\zeta|^2}dr\\&\leq C \left(1+d_x^{-1} \right)^{\frac13+\epsilon} \left(1+\frac1{|\zeta|} \right) .
\end{split}\end{equation}
If we do not want the singularity in the expression, we can have
\begin{equation}\begin{split}
|D_v^5|&\leq C(1+d_x^{-1})^{\frac43+\epsilon}.
\end{split}\end{equation}

To deal with the rest of terms, we need to discuss the derivative of $\tau_-$ and $p$ with respect to $\zeta$.

\begin{prop}
Suppose that $\Omega$ is a $C^1$ bounded convex domain in $\mathbb{R}^3$, $ x\in \Omega$, and $\zeta$, $\zeta'\in \mathbb{R}^3$. Then,
\begin{align}
\left| \frac{\partial}{\partial \zeta_i} \tau_-(x,\zeta) \right|&\leq \frac{\tau_-(x,\zeta)}{N(x,\zeta)|\zeta|}\\
\left| \frac{\partial}{\partial \zeta_i} p(x,\zeta) \right|&\leq \tau_-(x,\zeta) \left( 1+\frac1{N(x,\zeta)|\zeta|} \right).
\end{align}
\end{prop}
The above proposition is a direct consequence from the explicit formula in   Lemma 2 in \cite{GuoContinuity}.

Let \begin{equation}
\eta=\frac{\frac{\partial}{\partial \zeta_i}p(x,\zeta)}{\left|\frac{\partial}{\partial \zeta_i}p(x,\zeta)\right|}.\end{equation}
Then,
\begin{equation}
\nabla^x_{{\frac{\partial}{\partial \zeta_i}p(x,\zeta)}}f(p(x,\zeta),\zeta)=\nabla^x_\eta f(p(x,\zeta),\zeta)|{\frac{\partial}{\partial \zeta_i}p(x,\zeta)}|\leq C|{\frac{\partial}{\partial \zeta_i}p(x,\zeta)}|.
\end{equation}
Notice that \begin{equation}
   e^{-\nu(|\zeta|)\tau_-(x,\zeta)}\frac{\tau_-(x,\zeta)}{N(x,\zeta)|\zeta|}\leq Cd_x^{-1}.\end{equation}
We conclude the lemma.

\end{document}